\documentclass[3p]{elsarticle}

\usepackage[numbers]{natbib}

\usepackage{graphicx}
\usepackage[tableposition=top]{caption}
\usepackage{subcaption}

\usepackage[T1]{fontenc}
\usepackage[utf8]{inputenc}

\usepackage{adjustbox}

\usepackage{amsmath}
\usepackage{amsthm}
\usepackage{siunitx}
\usepackage{amsthm}
\usepackage{todonotes}
\usepackage{comment}
\usepackage[colorlinks=true, allcolors=blue]{hyperref} 

\usepackage{amssymb}             

\usepackage[bbgreekl]{mathbbol}

\DeclareSymbolFontAlphabet{\mathbbm}{bbold}
\DeclareSymbolFontAlphabet{\mathbb}{AMSb}%

\usepackage{Macros/mydef}
\usepackage{Macros/remarks}
\usepackage{bbm}
\usepackage{xspace}
\usepackage{algorithm}
\usepackage[algo2e]{algorithm2e} 
\usepackage{algpseudocode}
\usepackage{enumitem}
\usepackage{textcomp,mathcomp}
\usepackage{longtable}
\usepackage{multirow}
\usepackage{pdflscape}
\usepackage{mathtools}

\let\cite=\citet

\newtheorem{proposition}{Proposition}
\numberwithin{proposition}{section}

\newtheorem{theorem}{Theorem}
\numberwithin{theorem}{section}

\theoremstyle{definition}
\newtheorem{remark}{Remark}
\numberwithin{remark}{section}

\ifpdf
\hypersetup{
  pdftitle={FEMocean},
  pdfauthor={C.~Helanow, J.~Ahlkrona, ...}
}
\fi

\newcommand{\psu}{\unit{\gram \per \kilogram}}

\newcommand{\ind}{\ensuremath{\sigma}}
\newcommand{\proj}{\ensuremath{\Pi_h}}
\newcommand{\projS}{\ensuremath{\Pi_\Delta}}

\newcommand{\uh}{\ensuremath{\bu_h}}

\newcommand{\Ph}{\ensuremath{P_h}}

\newcommand{\w}{\mathbf{ w}}

\newcommand{\grad}{\GRAD}

\newcommand{\trn}[1]{{\left\vert\kern-0.25ex\left\vert\kern-0.25ex\left\vert #1 
    \right\vert\kern-0.25ex\right\vert\kern-0.25ex\right\vert}}

\newcommand{\fgref}[1]{Fig.~\ref{#1}}
\newcommand{\Fgref}[1]{Fig.~\ref{#1}}

\renewcommand{\eqref}[1]{~(\ref{#1})}
\newcommand{\Eqref}[1]{~(\ref{#1})}

\journal{Journal of Computational Physics}

\begin{document}
\date{}
\begin{frontmatter}
\title{A potential energy conserving finite element method for turbulent variable density flow: application to glacier-fjord circulation}
\author[1,2,3]{Lukas Lundgren\corref{cor1}}
\ead{lukas.lundgren@math.su.se}
\author[1,2,3]{Christian Helanow}
\ead{christian.helanow@proton.me}
\author[2,4]{Jonathan Wiskandt}
\ead{jonathan.wiskandt@misu.su.se}
\author[2,4]{Inga Monika Koszalka}
\ead{inga.koszalka@misu.su.se}
\author[1,2,3]{Josefin Ahlkrona}
\ead{ahlkrona@math.su.se}
\cortext[cor1]{Corresponding author}
\address[1]{Department of Mathematics, Stockholm, Sweden}
\address[2]{Bolin Centre for Climate Research, Stockholm University, Stockholm, Sweden}
\address[3]{Swedish e-Science Research Centre, Sweden}
\address[4]{Department of Meteorology, Stockholm University, Stockholm, Sweden}


\begin{abstract}
We introduce a continuous Galerkin finite element discretization of the non-hydrostatic Boussinesq approximation of the Navier-Stokes equations, suitable for various applications such as coastal ocean dynamics and ice-ocean interactions, among others. In particular, we introduce a consistent modification of the gravity force term which enhances conservation properties for Galerkin methods without strictly enforcing the divergence-free condition. We show that this modification results in a sharp energy estimate, including both kinetic and potential energy. Additionally, we propose a new, symmetric, tensor-based viscosity operator that is especially suitable for modeling turbulence in stratified flow. The viscosity coefficients are constructed using a residual-based shock-capturing method and the method conserves angular momentum and dissipates kinetic energy.  We validate our proposed method through numerical tests and use it to model the ocean circulation and basal melting beneath the ice tongue of the Ryder Glacier and the adjacent Sherard Osborn fjord in two dimensions on a fully unstructured mesh. Our results compare favorably with a standard numerical ocean model, showing better resolved turbulent flow features and reduced artificial diffusion.
\end{abstract}
\end{frontmatter}




%
%


\section{Introduction}


%










Variable density flow is common in industry, appearing in natural convection flows within heat exchangers, cooling systems and thermal insulators \citep{MOJTABATABARHOSEINI202339}. It is also prevalent in nature; in fact, numerical simulations of variable density flow are central to climate models that can help us predict and prepare for the effects of global warming \citep{IPPC_2023}. Improved models are specifically needed for turbulent gravity-driven flows in the climate system, such as the melt-water driven ocean circulation near the Greenland and Antarctic ice sheets \citep{IPCCspecialreport}. The rate at which the ice sheet melts is influenced by ocean velocity, temperature and salinity at the ice-ocean interface. In turn, these properties are significantly affected by the melt rate itself. Such non-linear feedback between glacial ice dynamics and fjord circulation has the potential to be an important contributing factor to changes in ice sheet mass loss and thus rising sea levels \citep{IPCCspecialreport,NowickiSeroussi,IPPC_2023,Nowicki_etal2016,Holland2008b,Asay-Davis_etal2016,Fyke,tc-14-3071-2020,tc-14-3033-2020}.

Many simulation software for variable density flow and fluid dynamics in general - including atmospheric and ocean models - rely on the finite volume (FV) or finite difference (FD) method rather than the finite element method (FEM). This is largely due to the lack of conservation properties of traditional FEMs. However, FEM offers several attractive features such as a sound mathematical framework and the ability to accurately represent complex geometries. Properly representing flows in the presence of complex geometry is one of the major challenges in FD and FV climate models, where the use of terrain-following or partial cell techniques can introduce errors \citep{Doos2022,piggott,Kimura_etal2013}. 


Discontinuous Galerkin (DG) methods are often suggested to address the issues with conservation of standard FEMs, both in fluid dynamics in general and for the coastal ocean and meltwater-driven circulation in particular. In \citep{karna} a DG hydrostatic model is developed for the coastal ocean and in \citep{Kimura_etal2013} and \citep{SCOTT2023102178} DG discretizations of a non-hydrostatic Boussinesq approximation of the Navier-Stokes equations are presented and used to model idealized configurations of the interaction between melting ice sheets and ocean circulation. 
A drawback of DG methods is that they tend to be associated with a higher computational cost and can, in some software, be more cumbersome to implement compared to the more traditional $H^1$-conforming continuous Galerkin (CG) FEM methods. 

However, recent progress in improving conservation properties for incompressible flow problems using $H^1$-conforming FEM opens up the possibility of CG models of density-driven flow. For $H^1$-conforming FEM, it is common for the divergence-free condition to be only satisfied weakly, which has consequences for stability and conservation properties. Recently, \cite{Charnyi2017} derived a formulation of the convective term that enables Galerkin methods to be kinetic energy, momentum and angular momentum conserving (EMAC). This formulation has been extended to be more efficient \citep{Charnyi2019}, further analyzed by \cite{Olshanskii2020,Olshanskii2024}, and extended to projection methods by \cite{Ingimarson2023}.  In particular, \citep{Olshanskii2024} showed that $H^1$-conforming FEM also satisfies local conservation properties. The EMAC formulation consistently enhances the accuracy and reliability of simulations compared to other approaches, especially over long time intervals. Motivated by the success of these methods, \cite{lundgren2024fully} extended the EMAC formulation to the variable density setting and derived a formulation that is shift-invariant, mass, squared density, kinetic energy, momentum and angular momentum conserving (SI-MEDMAC). Again, this was demonstrated to be more accurate and robust compared to other formulations in the literature. 

These new types of CG methods are relevant to variable density flow in general. For variable density, gravity-driven, flow specifically there is a balance between kinetic and potential energy, and it is beneficial for a numerical method to capture this balance as it improves accuracy and robustness \citep{Shen2024}. In this paper, we extend the SI-MEDMAC formulation so that potential energy is also conserved and we demonstrate its capabilities on a simulation of ocean circulation in the Sherard Osborn fjord adjacent to the Ryder Glacier, Greenland.  The circulation near the ice is driven by buoyant meltwater rising along the ice surface, creating a turbulent plume (\fgref{fig:CirculationMelt}). 




 Since the plume is highly turbulent, small-scale processes cannot be adequately resolved on a computational mesh. In the ocean modeling community, it is common to use tensor-based eddy viscosity \citep{Kimura_etal2013,SCOTT2023102178,Wiskandt,mitgcm_GOLDBERG,mitgcm_jordan} as a sub-grid scale model. This approach employs different viscosity coefficients in each coordinate direction, accommodating the significant variation in flow scales and mesh sizes between the horizontal and vertical directions. This enables a more accurate representation of these types of flow compared to scalar eddy viscosity. However, the tensors used in the references are not symmetric, resulting in the loss of angular momentum conservation \citep{lundgren2024fully,dao2024}.

We introduce a new symmetric viscosity tensor that conserves angular momentum and dissipates kinetic energy. To automatically construct the viscosity coefficients and maintain stability and accurate representation of small-scale structures, we scale the viscosity coefficients based on the residual of the PDE, a method commonly known as the residual viscosity (RV) method \citep{Nazarov_2013,Nazarov_Hoffman_2013,Stiernstrom2021,Lu2019,Marras2015,Dao_Nazarov_2022,lundgren_2024RV,Tominec2023}. Since the residual is small in smooth, resolved regions of the domain, high-order accuracy is preserved. For a comparison between the RV method and the well-known Lilly-Smagorinsky model, we refer to \cite{Marras2015}. For additional suppression of small-scale oscillations and improved accuracy in smooth regions, we also propose adding the high-order dissipation operator proposed by \cite{kuzmin2023_weno} which is a variation of local-projection stabilization \citep{Braack2006} and the variational multiscale (VMS) method \citep{John2006}. We note that the RV method is closely related to the entropy viscosity method \citep{Guermond_Pasquetti_2011}: for example, for scalar conservation laws, the residual is a special case of the entropy residual, since the solution to the equation can be chosen as an entropy functional. However, a similar proof of the convergence of the RV method for general non-linear scalar conservation laws, as in \citep{Nazarov_2013}, is not known for the entropy viscosity method. In addition, it can be difficult to choose the entropy functional so that the corresponding entropy viscosity method is robust.

Finally, to account for the boundary between the fjord and the surrounding open ocean, we also derive a modified directional do-nothing condition \citep{ddn} which is consistent with the new conservative formulation. The boundary condition is to our knowledge the first condition of its kind for FEM models of ice-ocean interaction. 

 The FEM is validated in two ways: 1) via benchmark problems to demonstrate convergence rates and energy conservation and 2) by simulating the circulation in an idealized 2D version of the Sherard Osborn Fjord, comparing our results with the high-resolution simulations of \citep{Wiskandt} which were produced with the FV state-of-the-art ocean-atmosphere model MITgcm (Massachusetts Institute of Technology general circulation model) \citep{Marshall,ADCROFT2004269}.



\subsection{Notation and preliminaries}
Let $\Omega \subset \mathbb{R}^d$ denote an open and bounded domain with Lipschitz boundary $\Gamma = \partial \Omega$. Furthermore $L^2(\Omega)$ denotes all $L^2$-integrable functions on $\Omega$ and  the $L^2(\Omega)$-inner product is written as $( \SCAL , \SCAL )$ with associated norm $\| \SCAL \|$. We define $(\GRAD \bu) := \partial_{x_i} u_j $ and $H^1(\Omega)$ the Sobolev space of functions with weak derivatives up to order one which are $L^2$-integrable.

We define $\be_i$ to be a unit vector pointing in the $i$th coordinate direction, \ie in 3D $\be_1 =  \hat{\bx} = (1, 0, 0)^\top$, $\be_2 =  \hat{\by}= (0, 1, 0)^\top$ and $\be_3 =  \hat{\bz}= (0, 0, 1)^\top$, and denote $\polI$ as the identity matrix of size $d$.

The following identities will be used repeatedly throughout this paper. 
Given the trilinear form 
\begin{equation} \label{eq:trilinear definition}
b( \bu, \bv, \bw) :=  (\bu \SCAL \GRAD \bv, \bw) = \l( ( \GRAD \bv)^\top \bu, \bw \r) = ( (\GRAD \bv) \bw , \bu)
\end{equation}
and that $\rho, w \in H^1(\Omega)$ and $\bu, \bw, \bv \in \bH^1(\Omega) := [H^1(\Omega)]^d$: provided that $\bu \in \bH^1_0(\Omega)$ the following identities hold
\begin{align}
\label{eq:IBP1} &b ( \bu , \bv, \bw ) = - b( \bu, \bw , \bv ) - ( (\DIV \bu) \bv , \bw),\\
\label{eq:IBP2} &b ( \bu , \bw, \bw ) =  - \frac{1}{2} ( (\DIV \bu) \bw , \bw),\\
\label{eq:advection_IBP1} & ( \bu \SCAL \GRAD \rho , w ) = -( \bu \SCAL \GRAD w, \rho ) - ( ( \DIV \bu ) \rho , w) , \\
\label{eq:advection_IBP2} & ( \bu \SCAL \GRAD \rho , \rho ) = - \frac{1}{2} ( ( \DIV \bu ) \rho , \rho ),
\end{align}
due to integration by parts.

The contraction operator, : , is defined as
\begin{equation}
A : B = \sum_{i,j = 1}^d A_{ij} B_{ij} = \text{tr} \l(A B^\top \r),
\end{equation}
for any two $d \times d$ matrices $A$ and $B$. We will use that a matrix and its transpose have the same trace and for any four matrices $A,B,C,D$ with matching dimensions, the trace operator fulfills the cyclic property:
\begin{align} 
\label{eq:trace_transpose}  \text{tr}(A) &= \text{tr}\l(A^\top \r),\\
\label{eq:cyclic} \text{tr}(A B C D) &= \text{tr}( B C D A).
\end{align}



\section{Model problem and a novel formulation}
\subsection{Governing equations}
\label{sec:model-problem}

We consider the Boussinesq approximation of the Navier-Stokes equations:
\begin{subequations}
  \begin{align}
    \p_t \bu + \bu \SCAL \GRAD \bu & = - \frac{\GRAD p}{\rho_0} + \DIV \l( \nu  \l( \GRAD \bu + (\GRAD \bu)^\top \r) \r) + \bef ,  \quad & (\bx, t ) \in \Omega \CROSS \l(  0,\widehat{t} \ \r ] , \label{eq:ns_momentum}\\ 
    \DIV \bu &=0, \quad &(\bx, t ) \in \Omega \CROSS \l( 0, \widehat{t} \ \r ], \label{eq:ns_mass}
  \end{align}\label{eq:ns}
\end{subequations}
where $\widehat{t}$ is the final time, $ \bu$ is the velocity field, $p$ is the pressure, $\nu$ is kinematic water viscosity, $\rho_0$ is the reference density and
 \begin{equation}
 \bef :=  -\frac{g\rho}{\rho_0} \hat{\by}
 \end{equation}
is the body force. Note that instead of using the standard approach of adding horizontal and vertical eddy viscosity to the model (as seen in \citep{Wiskandt, Kimura_etal2013, SCOTT2023102178}), we only include the kinematic water viscosity. Sub-grid scale effects are instead managed by our stabilization technique which acts as an implicit large eddy simulation (LES), detailed in Section \ref{sec:discretization}. We apply the same approach to the diffusivity of the tracer equations discussed below.


The body force depends on the constant of gravitational acceleration, $g$, and the density $\rho = \rho_0 + \delta\rho$, where $\delta \rho$ is a density perturbation around $\rho_0$. We here consider a temperature and salinity dependent density as relevant to ocean modeling. In particular, following the fjord-circulation model of \citep{Wiskandt} we use a linear equation of state 
\begin{equation}
  \label{eq:rho}
  \delta\rho = \rho_0 (-\alpha_T (T - T_0) + \beta_S (S - S_0)),
\end{equation}  
where $T_0$ and $S_0$ are the reference temperature and salinity, and $\alpha_T$ and $\beta_S$ are the thermal expansion and saline contraction coefficients. The evolution of the two tracers, temperature $T$ and salinity $S$, are described by the same general advection-diffusion equation. For a tracer $\phi$, we have
\begin{equation}\label{eq:tracer}
\frac{\partial \phi} {\partial t} + \bu \SCAL \GRAD \phi =  \GRAD\SCAL (\kappa_{\phi} \GRAD\phi) + \zeta(x) (\phi_{res}(y) - \phi), \quad (\bx, t ) \in \Omega \CROSS \l( 0, \widehat{t} \ \r ],
\end{equation}
where $\kappa_{\phi}$ is the physical diffusivity for each tracer (thermal or molecular). All parameter values relevant to glacier-fjord modeling are specified in Table \ref{tab:parameters}. The additional term on the right-hand side is a source term accounting for a restoring region, which will be described in detail in the following section. 

\begin{remark}
We do not consider the Coriolis force in this paper, relevant for large-scale ocean applications \citep{Doos2022}. For the numerical experiments on ice-ocean interaction, this is as in \citep{Wiskandt} motivated by the fact that the width of the Sherard Osborn fjord is about 9 km and rotational variabilities are small in comparison to along fjord changes.
\end{remark}

\subsection{Boundary and initial conditions}\label{sec:bcs}

We here state the boundary and initial conditions that are relevant for ice-ocean interaction that will be used in the simulations of the Sherard Osborn fjord in Section \ref{sec:numexp}. In certain experiments designed to support the theoretical analysis, simplified conditions are used and will be specified where applicable.

The boundary of the domain, $\Gamma$, is divided into the ice-ocean interface, $\Gamma_{ice}$, atmosphere-ocean interface, $\Gamma_{atm}$, the ocean floor, $\Gamma_{b}$, the interfaces to the open ocean, $\Gamma_{oc}$, and a small section beneath the ice grounding line, $\Gamma_{gl}$. The boundary $\Gamma_{gl}$ is introduced for the sole purpose of mimicking the setup in \citep{Wiskandt}. 
 
We consider a rigid-lid approximation with no-slip boundary conditions for the velocity at each boundary except at the open ocean boundary we apply a modified directional do-nothing condition \citep{ddn}. The non-homogeneous boundary conditions for the velocity and pressure can be summarized as  

\begin{subequations}
  \begin{align}
      \bu  &= 0, \quad    &\text{on }  \Gamma_{atm} \cup \Gamma_{ice} \cup \Gamma_{b} \cup \Gamma_{gl},\label{eq:bc_drag}\\
       \label{eq:ddn}
        \l( \nu  \l( \GRAD \bu +  \l(\GRAD \bu  \r)^\top  \r) -  \frac{p}{\rho_0} \polI \r)  \bn - \frac{1}{2} ( \bu \SCAL \bn )_- \bu &= -p_{hyd}  \bn , \quad \text{on } \Gamma_{oc}, 
  \end{align}
  \label{eq:ocean_bcs}
\end{subequations}
where $(\bn \SCAL \bu)_\pm = |\bn \SCAL \bu| \pm \bn \SCAL \bu $ and
\begin{equation}
  p_{hyd} = C + \int_{y_{min}}^{y}  \bef \SCAL \hat{\by} dy,
  \end{equation}
is the hydrostatic pressure and $C$ is an integration constant.
The modified directional do-nothing condition \eqref{eq:ddn} models the connection between the fjord and the surrounding open ocean and is based on the original directional do-nothing conditions \citep{ddn}. Rather than assuming a stress-free boundary as in \citep{ddn}, we apply a hydrostatic pressure from the open ocean so that the fjord water will be at rest if there is no melting or initial flow. The integration constant $C$ is set so that the hydrostatic pressure is $0$ at the atmosphere (top). 

For the tracers $\phi=T$ or $\phi=S$, an insulating homogenous Neumann condition is applied at $\Gamma_{b}$, $\Gamma_{gl}$ and $\Gamma_{atm}$. At the ice-ocean interface $\Gamma_{ice}$ a tracer flux $F_{\phi}$ is prescribed which represents meltwater input via the so-called virtual salt flux formulation. These fluxes are dependent on the ocean (and ice) state and are modeled by the commonly used three-equation model as described in Section \ref{sec:numexp}. At the open ocean boundary, a Dirichlet boundary condition is specified using a vertical tracer restoring profile $\phi_{res}$ which is defined in Section \ref{sec:numexp}. The non-homogeneous boundary conditions for the tracers can be summarized as
\begin{subequations}
  \begin{align}
    \phi &= \phi_{res}(y) &\text{on } \Gamma_{oc},\label{eq:tracer_bc_ocean}\\
    -\bn \SCAL (\kappa_{\phi}\GRAD\phi) &= 0 &\text{on } \Gamma_{b}, \Gamma_{gl}, \Gamma_{atm},\label{eq:tracer_bc_insulation}\\
    -\bn\SCAL(\kappa_{\phi}\GRAD\phi) &= F_{\phi}(\phi,\bu) &\text{on } \Gamma_{ice}, \label{eq:tracer_bc_flux}
  \end{align}
  \label{eq:tracer_bcs}
\end{subequations}
where $\bn$ is the to the boundary outward-pointing unit normal. 

The temperature and salinity are gradually restored to $\phi_{res}$ via the last term of \eqref{eq:tracer} over a horizontal region specified by the decay rate $\zeta(x)$ as in \citep{Asay-Davis_etal2016, Wiskandt}. If the $x$-position of $\Gamma_{oc}$ is denoted by $x_{oc}$, the decay rate is on the form
\begin{equation}
  \label{eq:restoring}
  \zeta(x) = \zeta_0 \max\left(0, \frac{(x - (x_{oc}-x_r))}{x_r}\right),
\end{equation}
where $\zeta_0$ is a frequency (days$^{-1}$) and $x_r$ is the width of the restoring region. That is, the decay rate is a function that is zero for all $x < x_{oc}-x_r$ and linearly increases from $0$ at $x_{oc}-x_r$ to $\zeta_0$ at the open ocean boundary.

The vertical restoring tracer profile also serves as initial conditions $\phi(\bx,t=0)=\phi_{res}(y)$ for the simulations, together with a zero initial velocity.   


\subsection{A novel conservative formulation}
The momentum equations \eqref{eq:ns_momentum} can be equivalently written as
\begin{equation}
  \label{eq:ns_Boussinesq}
  \p_t \bu  + \bu \SCAL \GRAD \bu  = - \GRAD \tilde{p} + \GRAD \SCAL \l( \nu  \l( \GRAD \bu + (\GRAD \bu )^\top \r) \r)  - \frac{g \delta \rho}{\rho_0} \Hat{\by} ,
\end{equation}
leading to a modified pressure $\tilde{p} = \frac{1}{\rho_0} p + g y$. This modification, aimed at eliminating cancellation errors resulting from $\rho_0 + \delta \rho$, is frequently used in the Boussinesq and ocean modeling literature. Next, we consider the energy, momentum and angular momentum conserving (EMAC) formulation \citep{Charnyi2017} 
\begin{equation}
  \label{eq:ns_EMAC}
  \p_t \bu  + \bu \SCAL \GRAD \bu + (\GRAD \bu) \bu + (\DIV \bu) \bu = - \GRAD \tilde{\tilde{p}} + \GRAD \SCAL \l( \nu  \l( \GRAD \bu + (\GRAD \bu )^\top \r) \r)  - \frac{g \delta \rho}{\rho_0} \Hat{\by} ,
\end{equation}
where $\tilde{\tilde{p}} = \frac{1}{\rho_0} p - \frac{1}{2} \bu \SCAL \bu + g  y$ is a further modified pressure. The EMAC formulation has favorable conservation properties when discretized by a Galerkin method leading to increased stability and accuracy \citep{Olshanskii2020, Charnyi2019, Ingimarson2023}. Similarly, the tracer equations can be written as
\begin{equation}\label{eq:tracer_conservative}
  \partial_t \phi  + \bu \SCAL \GRAD \phi + \frac{1}{2} (\DIV \bu) \l( \phi - \overline{ \phi } \r) =  \GRAD\SCAL (\kappa_{\phi} \GRAD\phi) + \zeta(x) (\phi_{res}(y) - \phi),
\end{equation}
where $\overline{ \phi } := \frac{1}{|\Omega|} \int_\Omega \phi \ud \bx$ is the mean of the tracer field. The formulation in \label{eq:tracer_conservative} was introduced by \cite{lundgren2024fully} and can be shown to be shift-invariant in the sense that
\begin{equation} \label{eq:shift-invariant}
  \begin{split}
  \partial_t \phi  + \bu \SCAL \GRAD \phi + \frac{1}{2} (\DIV \bu) \l( \phi - \overline{ \phi } \r)& =  \GRAD\SCAL (\kappa_{\phi} \GRAD\phi), \\  &\Leftrightarrow  \\
  \partial_t ( \phi + c)  + \bu \SCAL \GRAD ( \phi + c) + \frac{1}{2} (\DIV \bu) \l( ( \phi + c) - \overline{ ( \phi + c) } \r)& =  \GRAD\SCAL (\kappa_{\phi} \GRAD( \phi + c)), \quad \forall c \in \mR.
  \end{split}
\end{equation}
The formulation \eqref{eq:tracer_conservative} also conserves tracer mass $\int_\Omega \phi \ud \bx$ and squared tracer density (tracer energy/variance) $\int_\Omega \phi^2 \ud \bx$ (for $\kappa_\phi = 0$) when discretized by a Galerkin method leading to increased stability and accuracy. Lastly, we introduce a formulation that, in combination with the reformulated tracer equations \eqref{eq:tracer_conservative}, is shift-invariant and tracer mass, potential energy, kinetic energy, squared tracer density, momentum and angular momentum conserving (SI-MEEDMAC):
\begin{equation}
  \label{eq:ns_EMAC_potential}
  \p_t \bu  + \bu \SCAL \GRAD \bu + (\GRAD \bu) \bu + (\DIV \bu) \bu = - \GRAD P + \GRAD \SCAL \l( \nu  \l( \GRAD \bu + (\GRAD \bu )^\top \r) \r) - \frac{g \delta \rho}{\rho_0} \Hat{\by} + \frac{1}{2} \GRAD \l( \frac{\delta \rho  }{\rho_0} g y  \r),
\end{equation}
leading to the modified pressure $ P = \frac{1}{\rho_0} p - \frac{1}{2} \bu \SCAL \bu + g y + \frac{1}{2} \frac{\delta \rho  }{\rho_0} g y $. As will be shown in Section \ref{Sec:stability_estimate}, the main novelty of the new formulation \eqref{eq:ns_EMAC_potential} compared to the EMAC formulation \eqref{eq:ns_EMAC} is that the new formulation also conserves potential energy. Such energy conservation is an attractive feature for ocean models and gravity-driven flows in general. The idea of deriving a formulation that conserves potential energy was inspired by \cite{Shen2024}, who achieved this using a different approach.

Due to the use of pressure shifts, the modified directional-do-nothing boundary condition \eqref{eq:ddn} must be equivalently rewritten as
\begin{equation} \label{eq:ddn_modified}
  \begin{split}
    \l( \nu  \l( \GRAD \bu +  \l(\GRAD \bu  \r)^\top  \r) - \l(P + \frac{1}{2} \bu\SCAL \bu - \frac{1}{2} \frac{\delta \rho  }{\rho_0} g y \r) \polI \r)  \bn - \frac{1}{2} ( \bu \SCAL \bn )_- \bu  = - P_{hyd}   \bn , \quad \text{on } \Gamma_{oc},
  \end{split}
\end{equation}
where $P_{hyd}$ is the modified hydrostatic pressure given by
\begin{equation}
  P_{hyd} = C + \int_{y_{min}}^{y}  - \frac{g \delta \rho}{\rho_0} \Hat{\by} \SCAL \hat{\by} dy,
  \end{equation}
where $C$ is chosen so that $p$ is zero on the top of the domain. This leads to the exact same energy estimate as in  \citep{ddn}. 

%






\section{Discretization and implicit LES model}
\label{sec:discretization}

\subsection{Galerkin discretization of the Boussinesq equations}
\label{sec:ns_discrete}

The discrete domain $\Omega_h$ is represented by a mesh $\mathcal{T}_h$. The mesh is composed of a finite number of disjoint elements $K$ so that $\Omega_h = \bigcup_{K\in\mathcal{T}_h} K$. We consider function spaces consisting of continuous functions that are piece-wise polynomial
\begin{equation}
  \label{eq:function_space}
  \polP_h^k(\mathcal{T}_h) = \left\{ v\in \mathcal{C}^0(\Omega_h): v|_K\in \polP_k \ \forall K  \in \mathcal{T}_h\right\},
\end{equation}
where $\polP_k$ denotes the space of multivariate polynomials of total degree at most $k$. We define the nodal based mesh size as $h(\bx) \in \calM_h$ using $L_2$-projection with additional smoothing
\begin{equation}
(h,w) + C_{\Delta} \l(|K|^{2/d} \GRAD h, \GRAD w \r) = \l( |K|^{1/d}/k, w \r), \quad \forall w \in \calM_h,
\end{equation}
where $|K|$ is the volume of the cell $K$ and $C_\Delta = \mathcal{O}(1)$ is a user-defined constant.
For the velocity, pressure and tracers we define the following function spaces
\begin{subequations}
  \label{eq:function_spaces}
  \begin{align}
    \bcalV_h  := \l[ \polP^{k+1}_h \r]^d, \quad
    Q_h & := \polP^k_h, \quad
    \calM_h  := \polP^k_h,
  \end{align}
\end{subequations}
which should be taken to include the Dirichlet boundary conditions in \eqref{eq:ocean_bcs} and \eqref{eq:tracer_bcs}.

This choice is compliant with the \emph{inf-sup} or LBB condition \citep{Babuska1973,Brezzi1974}. We also set the polynomial degree of the tracers to be the same as that of the pressure, as this automatically ensures mass conservation of the tracers \citep{lundgren2024fully}. Additionally, the relationship between $Q_h$ and $\calM_h$ has implications for the FEM's ability to maintain hydrostatic balance
\begin{equation}
\frac{\partial p}{\partial y} = -\rho(T,S) g,
\end{equation}
which is crucial in areas away from the flow, for example. Allowing the polynomial degree of the pressure to be one order higher than that of the tracers $T$ and $S$ helps maintain hydrostatic balance, which is particularly important on unstructured meshes \citep{Ford_etal2004}. In \citep{SCOTT2023102178} and \citep{Kimura_etal2013}, hydrostatic balance is achieved by choosing continuous $\polP_2$ elements for pressure and discontinuous $\polP_1$ elements for the tracers and velocity. In contrast, \citep{Ford_etal2004} handles this by solving for the baroclinic pressure gradients separately using a Discontinuous Galerkin method. Instead of lowering the polynomial degree of the tracers, which would result in larger dispersion errors, we propose using a different formulation of the momentum equations \eqref{eq:ns_EMAC_potential}. This approach leads to a sharp energy estimate that includes potential energy, and in Section \ref{Sec:noflow} , we numerically demonstrate that it reduces errors related to the satisfaction of hydrostatic balance.

The semi-discrete finite element approximation of the Boussinesq equations is derived by taking an inner product between the conservative formulation (\eqref{eq:tracer_conservative}, \eqref{eq:ns_EMAC_potential}  and \eqref{eq:ns_mass}) and the test functions from the function spaces \eqref{eq:function_spaces}. Integration by parts is performed on the pressure, viscous and diffusive terms and the boundary conditions \eqref{eq:bc_drag}, \eqref{eq:tracer_bcs}, 
\eqref{eq:ddn_modified} are applied. The FEM reads: Find $(\phi_h, \uh, \Ph)\in \calM_h \CROSS \bcalV_h \CROSS Q_h$ such that:

\begin{subequations} \label{eq:galerkin}
  \begin{equation}\label{eq:tracer_galerkin}
    \begin{split}
      (\p_t \phi_h  + \bu_h \SCAL \GRAD \phi_h + \frac{1}{2} (\DIV \bu_h) \l( \phi_h - \overline{ \phi_h } \r) , w) + (\kappa_{\phi} \GRAD\phi_h , \GRAD w) \\ + ( F_{\phi,h} , w )_{\Gamma_{ice}}  =    (\zeta(x) (\phi_{res}(y) - \phi_h) , w) , \quad \forall w \in \calM_h,
    \end{split}
     \end{equation}
  \begin{equation}
    \begin{split}
    ( \p_t \bu_h + \bu_h \SCAL \GRAD \bu_h + (\GRAD \bu_h) \bu_h + (\DIV \bu_h) \bu_h , \bv ) -( P_h, \DIV \bv) + \l(  \nu  \l( \GRAD \bu_h + (\GRAD \bu_h)^\top \r) , \GRAD \bv \r) & =  \\  - \l(   \frac{g \delta \rho_h}{\rho_0} \Hat{\by} , \bv \r)  -  \frac{1}{2} \l(   \frac{\delta \rho_h  }{\rho_0} g y  , \DIV \bv \r) + \frac{1}{2} \l(  ( \bu_h \SCAL \bn )_- \bu_h , \bv \r)_{\Gamma_{oc}} + \frac{1}{2} \l(  ( \bu_h \SCAL \bn ) \bu_h , \bv \r)_{\Gamma_{oc}} \\   - (P_{hyd} , \bv \SCAL \bn )_{\Gamma_{oc}}  ,  \quad  \forall \bv \in \bcalV_h, \\ 
    \label{eq:ns_galerkin}
  \end{split}
  \end{equation}
  \begin{equation}
    ( \DIV \bu_h , q ) = 0, \quad \forall q \in Q_h.
  \end{equation}
\end{subequations}

\subsubsection{Stabilized method acting as an implicit LES method}
It is well known that the Galerkin discretizations are unstable for convection-dominated problems, \ie $\kappa_{\phi}, \nu \ll h | \bu | $, and therefore the FEM \eqref{eq:galerkin} needs to be stabilized. We propose using a combination of residual-based artificial viscosity \citep{Nazarov_2013,Nazarov_Hoffman_2013,Stiernstrom2021,Lu2019,Marras2015,Dao_Nazarov_2022,lundgren_2024RV,aronson2023stabilized} and high-order dissipation based on Kuzmin et al.'s \citep{kuzmin2023_weno,Kuzmin2020} variant of the variational multiscale (VMS) method \citep{John2006,Braack2006}. 




We extend the RV method from scalar artificial viscosity to tensor-based artificial viscosity. This extension enables automatic tuning of the artificial viscosity coefficients to maintain stability and accurately represent small-scale structures without requiring user tuning. The tensor-based version of the RV method inherently accounts for the reduced mixing in the vertical direction compared to the horizontal direction, eliminating the need to explicitly introduce separate horizontal and vertical eddy viscosity or diffusivity to address stratification.

The tracer equation is stabilized by adding a blend of tensor artificial viscosity and tensor high-order dissipation
\begin{equation}
  \DIV ( \bbkappa_{\phi,h} \GRAD \phi_h  +   \bbkappa_{\phi,h,vms}( \GRAD \phi_h - \proj \GRAD \phi_h ) ),
\end{equation}
where $\bbkappa_{\phi,h} = \text{diag}( \kappa_{\phi,h,1}, \dots, \kappa_{\phi,h,d} )$, $\bbkappa_{\phi,h,vms} = \text{diag}( \kappa_{\phi,h,vms,1}, \dots, \kappa_{\phi,h,vms,d} )$ are diagonal matrices with scalar mesh-dependent coefficients $\kappa_{\phi,h,i}, \kappa_{\phi,h,vms,i} \geq 0$ to be defined. Following \citep{kuzmin2023_weno}, $\proj \grad \phi_h$ is defined as the following $L_2$ projection of $\grad \phi_h$: Find $\proj \grad \phi_h \in [\calM_h]^d$ such that:
\begin{equation} \label{eq:projection_LPS}
  ( \bbkappa_{h,vms}  \proj \grad \phi_h, \bw) = (\bbkappa_{h,vms} \grad \phi_h, \bw) \quad \forall \bw \in [\calM_h]^d.
\end{equation}

The momentum equations are stabilized in a similar fashion by adding
\begin{equation} \label{eq:momentum_tensor_viscosity}
  \begin{split}
   \GRAD ( \gamma_h \DIV \bu_h ) + \DIV \Biggl (  \sqrt{\bbnu_h}    \l( \GRAD \bu_h + (\GRAD \bu_h)^\top \r)  \sqrt{\bbnu_h}  \Biggr .  \\ \Biggl .  +  \sqrt{ \bbnu_{h,vms} }  \l( \GRAD \bu_h + (\GRAD \bu_h)^\top - \proj \GRAD \bu_h - (\proj \GRAD \bu_h)^\top  \r) \sqrt{ \bbnu_{h,vms} }  \Biggr ),
  \end{split}
\end{equation}
where $\gamma_h\geq 0$ is a grad-div stabilization \citep{Olshanskii_2004} coefficient to be defined and $\bbnu_{h} = \text{diag}( \nu_{h,1}, \dots, \nu_{h,d} )$, $\bbnu_{h,vms} = \text{diag}( \nu_{h,vms,1}, \dots, \nu_{h,vms,d} )$ are diagonal matrices with scalar mesh-dependent coefficients $\nu_{h,i}, \nu_{\phi,h,vms,i} \geq 0$ to be defined. We note that $ \bbnu_{h,vms} = \sqrt{\bbnu_{h,vms}} \sqrt{\bbnu_{h,vms}}$ with $\sqrt{\bbnu_{h,vms}} = \text{diag} \l( \sqrt{\nu_{h,vms,1}} , \dots,  \sqrt{\nu_{h,vms,d}} \r)$. $\proj \grad \bu_h$ is defined as the following $L_2$ projection of $\grad \bu_h$: Find $\proj \grad \bu_h \in [\bcalV_h]^d $ such that:
\begin{equation} \label{eq:projection_LPS_tensor}
  \l( \sqrt{\bbnu_{h,vms}} \proj \grad \bu_h \sqrt{\bbnu_{h,vms}}, \polV \r) = \l( \sqrt{\bbnu_{h,vms}} \grad \bu_h \sqrt{\bbnu_{h,vms}}, \polV \r) \quad \forall \polV \in [\bcalV_h]^d,
\end{equation}
where $[\bcalV_h]^d$ is a tensor space.

The stabilized FEM then reads: Find $(\phi_h, \uh, \Ph)\in \calM_h \CROSS \bcalV_h \CROSS Q_h$ such that:
\begin{subequations} \label{eq:stabilized fem}
  \begin{equation}\label{eq:tracer_stabilized}
    \begin{split}
      \l( \p_t \phi_h  + \bu_h \SCAL \GRAD \phi_h + \frac{1}{2} (\DIV \bu_h) \l( \phi_h - \overline{ \phi_h } \r) , w \r) + \l( (\kappa_{\phi} \polI + \bbkappa_{\phi,h} ) \GRAD\phi_h + \bbkappa_{\phi,h, vms} ( \GRAD \phi_h - \proj \GRAD \phi_h ) , \GRAD w \r) \\ + ( F_{\phi,h} , w )_{\Gamma_{ice}}  =    (\zeta(x) (\phi_{res}(y) - \phi_h) , w) , \quad \forall w \in \calM_h,
    \end{split}
     \end{equation}
  \begin{equation} \label{eq:momentum_stabilized}
    \begin{split}
    ( \p_t \bu_h + \bu_h \SCAL \GRAD \bu_h + (\GRAD \bu_h) \bu_h + (\DIV \bu_h) \bu_h , \bv ) -( P_h, \DIV \bv) + (\gamma_h \DIV \bu_h, \DIV \bv) \\ + \Biggl (   \nu    \l( \GRAD \bu_h + (\GRAD \bu_h)^\top \r) +   \sqrt{\bbnu_h}  \l( \GRAD \bu_h + (\GRAD \bu_h)^\top \r)  \sqrt{\bbnu_h} \Biggr .  \\ \Biggl .  +  \sqrt{ \bbnu_{h,vms} }  \l( \GRAD \bu_h + (\GRAD \bu_h)^\top - \proj \GRAD \bu_h - (\proj \GRAD \bu_h)^\top  \r) \sqrt{ \bbnu_{h,vms} } , \GRAD \bv \Biggr ) & =  \\  -  \l(   \frac{g \delta \rho_h}{\rho_0} \Hat{\by} , \bv \r)  -  \frac{1}{2} \l(   \frac{\delta \rho_h  }{\rho_0} g y  , \DIV \bv \r) + \frac{1}{2} \l(  ( \bu_h \SCAL \bn )_- \bu_h , \bv \r)_{\Gamma_{oc}} + \frac{1}{2} \l(  ( \bu_h \SCAL \bn ) \bu_h , \bv \r)_{\Gamma_{oc}} \\   - (P_{hyd} , \bv \SCAL \bn )_{\Gamma_{oc}}   ,  \quad  \forall \bv \in \bcalV_h, 
  \end{split}
  \end{equation}
  \begin{equation} \label{eq:divergence_free_stabilized}
    ( \DIV \bu_h , q ) = 0, \quad \forall q \in Q_h.
  \end{equation}
\end{subequations}
Provided that $\bbkappa_{\phi,h}, \bbnu_h$ are constructed to be sufficiently large, \eqref{eq:stabilized fem} will be stable whereas $\gamma_h$ provides additional divergence cleaning and $\bbkappa_{\phi,h,vms}, \bbnu_{h,vms}$ provides additional high-order dissipation to suppress small-scale oscillations. We set
\begin{equation}
  \begin{split}
    \kappa_{\phi,h, i} &= \ind_{\phi,h} C_{\text{max}} h |\bu_{h,i}|, \\ 
    \kappa_{\phi,h, vms, i} &= (1 -\ind_{\phi,h} ) C_{\text{max,vms}} h |\bu_{h,i}|, \\ 
    \nu_{h,i} &=  \ind_{\bu,h} C_{\text{max}} h |\bu_{h,i}|, \\
    \nu_{h,vms,i} &=  (1 - \ind_{\bu,h}) C_{\text{max,vms}} h |\bu_{h,i}|, \\
    \gamma_{h} &=   C_{\text{max}} h |\bu_h|,
  \end{split}
\end{equation}
where $C_\text{max}, C_\text{max,vms}$ are user-defined parameters determining the amount of first-order viscosity and high-order dissipation used and $\ind_{\phi,h}, \ind_{\bu,h}  \in  \calM_h  \cap [0, 1] $ are discontinuity indicators for $\phi_h$ and $\bu_h$ that are close to 0 in smooth regions and close to 1 in underresolved regions. For stratified flow, the stabilization will generally be lower in the vertical direction due to the typically lower vertical velocities. This is comparable to the reduced vertical eddy viscosity/diffusivity used in standard eddy-viscosity models, such as those described in \cite{Wiskandt}.

 The algorithm used to compute the indicators $\ind_h$ is quite technical and interested readers are referred to \ref{appendix:discontinuity_indicator} for further details. The technique (with $C_{\text{max,vms}} = 0$) was recently successfully applied to variable density incompressible flow in \citet{lundgren_2024RV} where it was demonstrated that the method is robust for large density ratios while still being high-order accurate in smooth regions. Recently, \cite{kuzmin2023_weno} proposed to stabilize convection-dominated FEM using a blend of low-order stabilization scaled with $\bbkappa_h, \bbnu_h$ and high-order stabilization scaled with $\bbkappa_{h,vms}, \bbnu_{h,vms}$. The difference between our approach and that of \citep{kuzmin2023_weno} lies in the construction of our discontinuity indicator, $\ind_h$, which is based on the PDE residual and utilizes tensor viscosity. In contrast, \citep{kuzmin2023_weno} employs a different method to construct the indicator and uses scalar viscosity. We note that Kuzmin et al.'s high-order dissipation operator \citep{kuzmin2023_weno, Kuzmin2020}, scaled with $C_\text{max,vms}$, provides further suppression of small-scale, high-frequency oscillations and enhances accuracy for smooth solutions.


 As detailed in \ref{Sec:symmetric_stabilization}, by following similar steps as performed by \cite[Sec 1]{John2006}, it is possible to show that \eqref{eq:tracer_stabilized}-\eqref{eq:momentum_stabilized} is equivalent to
  \begin{subequations} \label{eq:stabilized fem symmetric}
    \begin{equation}\label{eq:tracer_stabilized_symmetric}
      \begin{split}
        (\p_t \phi_h  + \bu_h \SCAL \GRAD \phi_h + \frac{1}{2} (\DIV \bu_h) \l( \phi_h - \overline{ \phi_h } \r) , w) + ( (\kappa_{\phi} \polI + \bbkappa_{\phi,h} ) \GRAD\phi_h , \GRAD w) \\ + ( \bbkappa_{\phi,h, vms}  ( \GRAD \phi_h - \proj \GRAD \phi_h ) , \GRAD w - \proj \GRAD w)  + ( F_{\phi,h} , w )_{\Gamma_{ice}}  =    (\zeta(x) (\phi_{res}(y) - \phi_h) , w) , \quad \forall w \in \calM_h  ,
      \end{split}
       \end{equation}
       \begin{equation} \label{eq:momentum_stabilized_symmetric}
        \begin{split}
        ( \p_t \bu_h + \bu_h \SCAL \GRAD \bu_h + (\GRAD \bu_h) \bu_h + (\DIV \bu_h) \bu_h , \bv ) -( P_h, \DIV \bv) + (\gamma_h \DIV \bu_h, \DIV \bv) \\ + \frac{1}{2} \l(   \nu    \l( \GRAD \bu_h + (\GRAD \bu_h)^\top \r) +   \sqrt{\bbnu_h}  \l( \GRAD \bu_h + (\GRAD \bu_h)^\top \r)  \sqrt{\bbnu_h}, \GRAD \bv + \GRAD \bv^\top \r)   \\   + \frac{1}{2} \l( \sqrt{ \bbnu_{h,vms} }  \l( \GRAD \bu_h + (\GRAD \bu_h)^\top - \proj \GRAD \bu_h - (\proj \GRAD \bu_h)^\top  \r)  \sqrt{ \bbnu_{h,vms} } , \GRAD \bv + \GRAD \bv^\top - \proj \GRAD \bv - \proj \GRAD \bv^\top \r) & =  \\   - \l(   \frac{g \delta \rho_h}{\rho_0} \Hat{\by} , \bv \r)  -  \frac{1}{2} \l(   \frac{\delta \rho_h  }{\rho_0} g y  , \DIV \bv \r) + \frac{1}{2} \l(  ( \bu_h \SCAL \bn )_- \bu_h , \bv \r)_{\Gamma_{oc}} + \frac{1}{2} \l(  ( \bu_h \SCAL \bn ) \bu_h , \bv \r)_{\Gamma_{oc}} \\   - (P_{hyd} , \bv \SCAL \bn )_{\Gamma_{oc}}  ,  \quad  \forall \bv \in \bcalV_h.
      \end{split}
      \end{equation}
  \end{subequations}


\begin{remark}
We note that the viscosity operators $ \bbnu \GRAD \bu $ \citep{Kimura_etal2013} and $ \bbnu \l( \GRAD \bu + (\GRAD \bu)^\top \r)$ \citep{SCOTT2023102178} are currently being used as eddy viscosity, acting as a sub-grid scale model in ocean modeling. The downside with these operators is that they are not symmetric and, therefore, do not conserve angular momentum. Additionally, it is unclear to the authors whether these operators are provably kinetic energy dissipative. The advantage of using $ \sqrt{ \bbnu } \l( \GRAD \bu + (\GRAD \bu)^\top \r) \sqrt{ \bbnu } $ in \eqref{eq:momentum_tensor_viscosity} is that it guarantees both angular momentum conservation and kinetic energy dissipation for any choice of viscosity coefficients inside $\bbnu$, as demonstrated in the next section.
\end{remark}

\section{Stability and conservation estimates} \label{Sec:stability_estimate}

In this section, we show that the stabilized FEM satisfies sharp stability and conservation estimates. To simplify the analysis, we set the restoring term to zero, i.e., $\zeta(x) = 0$. We also assume homogeneous Neumann boundary conditions for the tracers and no-slip conditions for velocity on $\p \Omega$
\begin{equation}
  \begin{split}
    \bn \SCAL( \kappa_\phi \GRAD  \phi) = 0 , \quad \p \Omega, \\
    \bu = 0, \quad \p \Omega.
  \end{split}
\end{equation}
This requires modifying the pressure space $Q_h$ to the zero average space $Q_{0,h} = \{ q \in Q_h : \int_\Omega q \ud \bx = 0 \}$.

The momentum equations \eqref{eq:ns_Boussinesq}, combined with the tracer equations \eqref{eq:tracer}, can be shown \citep[Equation (9)-(10)]{winters_1995} to satisfy the following kinetic- and potential energy evolution equations
\begin{align}
  \p_t  \frac{1}{2} \| \bu \|^2     + \frac{1}{2} \l \|  \nu^\frac{1}{2}  \l( \GRAD \bu + (\GRAD \bu)^\top \r)  \r \|^2    =  (- \frac{g \delta \rho}{\rho_0} \Hat{\by},\bu), \\
  \p_t  E_p    = -(- \frac{g \delta \rho}{\rho_0} \Hat{\by},\bu)   +  ( \alpha_T g  T \GRAD y  ,    \DIV ( \kappa_T \polI )  ) - ( \beta_S g  S_h \GRAD y  ,    \DIV  (\kappa_S \polI)  ) \\ -  ( \kappa_T g \alpha_T \bn \SCAL  \by , T  )_{\p \Omega } +  ( \kappa_S g \beta_S \bn \SCAL  \by , S_h  )_{\p \Omega }  \nonumber  ,
\end{align}
with kinetic energy defined as $\frac{1}{2} \| \bu \|^2$ and potential energy defined as $E_p = (g (y - \overline{y})    ,  -\alpha_T  T + \beta_S  S)$. Note that if $\kappa_{T,S}$ is constant, then $(\alpha_T g T \GRAD y, \DIV (\kappa_T \polI)) - (\beta_S g S_h \GRAD y, \DIV (\kappa_S \polI)) = 0$, indicating that thermal- and salinity diffusion affect potential energy only through the boundary terms. These equations can be combined into
  \begin{equation} \label{eq:potential_energy_model}
    \begin{split}
      \p_t \l( \frac{1}{2} \| \bu \|^2 +  E_p \r)   + \frac{1}{2} \l \|  \nu^\frac{1}{2}  \l( \GRAD \bu + (\GRAD \bu)^\top \r)  \r \|^2  \\  =  ( \alpha_T g  T \GRAD y  ,    \DIV ( \kappa_T \polI )  ) - ( \beta_S g  S_h \GRAD y  ,    \DIV  (\kappa_S \polI)  ) -  ( \kappa_T g \alpha_T \bn \SCAL  \by , T  )_{\p \Omega } +  ( \kappa_S g \beta_S \bn \SCAL  \by , S_h  )_{\p \Omega }    ,
    \end{split}
  \end{equation}
which shows the balance between these energies. 

Next, we show that our stabilized FEM satisfies a discrete counterpart of \eqref{eq:potential_energy_model} with additional artificial diffusion from $\bbkappa_h, \bbnu_h, \bbnu_{h,vms}, \gamma_h$. The novelty of our proposed FEM is that it can satisfy the full energy balance without strongly imposing the divergence-free constraint while also conserving angular momentum. The key ingredients enabling this are our new formulation \eqref{eq:ns_EMAC_potential} of the Boussinesq approximation and our new symmetric viscous tensor $\sqrt{\bbnu_h} \l( \GRAD \bu_h + (\GRAD \bu_h)^\top \r) \sqrt{\bbnu_h}$.

\begin{theorem} \label{proposition_potential}
 The stabilized FEM \eqref{eq:stabilized fem symmetric}, with no-slip boundary conditions for $\bu_h$ and homogeneous Neumann boundary conditions for the tracers, satisfies the following stability estimate
  \begin{equation}
    \begin{split}
      \p_t \l( \frac{1}{2} \| \bu_h \|^2 +  (g (y - \overline{y})    ,  -\alpha_T  T_h + \beta_S S_h) \r)  \\
       + \frac{1}{2} \l \|  \nu^\frac{1}{2}  \l( \GRAD \bu_h + (\GRAD \bu_h)^\top \r)   \r \|^2  
       + \frac{1}{2} \l \|   \bbnu_h^\frac{1}{4}  \l( \GRAD \bu_h + (\GRAD \bu_h)^\top \r) \bbnu_h^\frac{1}{4}  \r \|^2  + \l \| \sqrt{\gamma_h} \DIV \bu_h \r \|^2
       \\ + \frac{1}{2} \l \|  \bbnu_{h,vms}^\frac{1}{4}  \l( \GRAD \bu_h + (\GRAD \bu_h)^\top - \proj \GRAD \bu_h - ( \proj \GRAD \bu_h)^\top \r) \bbnu_{h,vms}^\frac{1}{4} \r \|^2 \\ =  - ( \DIV(  (\kappa_T \polI + \bbkappa_{T,h} )  \hat{\by} ) ,     g \alpha_T  T_h) + ( \DIV (  ( \kappa_S \polI + \bbkappa_{S,h} ) \hat{\by} ),    g \beta_S S_h    ) \\  +  ( \bn \SCAL ( (\kappa_T \polI + \bbkappa_{T,h})\hat{\by} ), g \alpha_T T_h )_{\p \Omega} - ( \bn \SCAL ( (\kappa_S \polI + \bbkappa_{S,h})\hat{\by} ), g \beta_S S_h )_{\p \Omega} .  
    \end{split}
  \end{equation}
\end{theorem}

\begin{proof}
  Setting $\bv = \bu_h, q = P_h$ inside the momentum equations \eqref{eq:momentum_stabilized_symmetric} and divergence-free constraint \eqref{eq:divergence_free_stabilized}, and adding them gives

  \begin{equation} \label{eq:step0_energy}
    \begin{split}
      \p_t \frac{1}{2} \| \bu_h \|^2 + 2 b(\bu_h,\bu_h,\bu_h) + ((\DIV \bu_h) \bu_h, \bu_h) + \l \| \sqrt{\gamma_h} \DIV \bu_h \r \|^2 \\ 
      + \frac{1}{2} \Biggl (   \nu    \l( \GRAD \bu_h + (\GRAD \bu_h)^\top \r)   + \sqrt{\bbnu_h}  \l( \GRAD \bu_h + (\GRAD \bu_h)^\top \r) \sqrt{ \bbnu_h } , \GRAD \bu_h + (\GRAD \bu_h)^\top \Biggr )  \\ 
      + \frac{1}{2} \Biggl ( \sqrt{ \bbnu_{h,vms}}  \l( \GRAD \bu_h + (\GRAD \bu_h)^\top - \proj \GRAD \bu_h - ( \proj \GRAD \bu_h)^\top \r) \sqrt{ \bbnu_{h,vms}}  \Biggr .\\ \Biggl . , \GRAD \bu_h + (\GRAD \bu_h)^\top - \proj \GRAD \bu_h - ( \proj \GRAD \bu_h)^\top \Biggr ) \\ 
      = \l( - \frac{g \delta \rho_h}{\rho_0} \Hat{\by}, \bu_h \r) - \frac{1}{2} \l(  \frac{\delta \rho_h  }{\rho_0} g y  , \DIV \bu_h \r),  
    \end{split}
  \end{equation}
where $ b(\bu_h, \bu_h, \bu_h) = ( \bu_h \SCAL \GRAD \bu_h, \bu_h ) = ( (\GRAD \bu_h) \bu_h , \bu_h )$ follows from the definition of the trilinear form \eqref{eq:trilinear definition}.

Next, using the cyclic property of the trace operator \eqref{eq:cyclic}, one can show that
\begin{equation} \label{eq:sym_dissipation}
  \begin{split}
  &\l(  \sqrt{ \bbnu_h }  \l( \GRAD \bu_h + \l( \GRAD \bu_h \r)^\top \r)  \sqrt{\bbnu_h} , \GRAD \bu_h + (\GRAD \bu_h)^\top \r)  \\
 =& \text{tr} \l(   \sqrt{ \bbnu_h }  \l( \GRAD \bu_h + \l( \GRAD \bu_h \r)^\top \r)  \sqrt{\bbnu_h} \l( \GRAD \bu_h + (\GRAD \bu_h)^\top \r) \r)  \\
 =& \text{tr} \l(    \bbnu_h^{\frac{1}{4}}    \l( \GRAD \bu_h + \l( \GRAD \bu_h \r)^\top \r) \bbnu_h^{\frac{1}{4}} \bbnu_h^{\frac{1}{4}} \l( \GRAD \bu_h + (\GRAD \bu_h)^\top \r) \bbnu_h^{\frac{1}{4}} \r)  \\
  = & \l(   \bbnu_h^{\frac{1}{4}}  \l( \GRAD \bu_h + \l( \GRAD \bu_h \r)^\top \r) \bbnu_h^{\frac{1}{4}},  \bbnu_h^{\frac{1}{4}} \l( \GRAD \bu_h + (\GRAD \bu_h)^\top \r) \bbnu_h^{\frac{1}{4}} \r) \\
   =& \l \|   \bbnu_h^{\frac{1}{4}}  \l( \GRAD \bu_h + \l( \GRAD \bu_h \r)^\top \r) \bbnu_h^{\frac{1}{4}} \r \|^2 .
  \end{split}
  \end{equation}

Applying \eqref{eq:sym_dissipation} on \eqref{eq:step0_energy} and performing similar steps on the other viscosity terms gives
\begin{equation} \label{eq:step1_energy}
  \begin{split}
    \p_t \frac{1}{2} \| \bu_h \|^2 + 2 b(\bu_h,\bu_h,\bu_h) + ((\DIV \bu_h) \bu_h, \bu_h) + \l \| \sqrt{\gamma_h} \DIV \bu_h \r \|^2 \\
    + \frac{1}{2} \l \|  \nu^\frac{1}{2}  \l( \GRAD \bu_h + (\GRAD \bu_h)^\top \r)   \r \|^2  
    + \frac{1}{2} \l \|  \bbnu_h^\frac{1}{4}  \l( \GRAD \bu_h + (\GRAD \bu_h)^\top \r) \bbnu_h^\frac{1}{4}  \r \|^2  \\ + \frac{1}{2} \l \|  \bbnu_{h,vms}^\frac{1}{4}  \l( \GRAD \bu_h + (\GRAD \bu_h)^\top - \proj \GRAD \bu_h - ( \proj \GRAD \bu_h)^\top \r) \bbnu_{h,vms}^\frac{1}{4} \r \|^2 \\
    = \l( - \frac{g \delta \rho_h}{\rho_0} \Hat{\by}, \bu_h \r) - \frac{1}{2} \l(  \frac{\delta \rho_h  }{\rho_0} g y  , \DIV \bu_h \r).
  \end{split}
\end{equation}

Using identity \eqref{eq:IBP2} and doing integration by parts on the last term in \eqref{eq:step1_energy} two times gives
  \begin{equation} \label{eq:proof_step0}
    \begin{split}
      \p_t \| \bu_h \|^2 +  \l \| \sqrt{\gamma_h} \DIV \bu_h \r \|^2   + \frac{1}{2} \l \|  \nu^\frac{1}{2}  \l( \GRAD \bu_h + (\GRAD \bu_h)^\top \r)   \r \|^2  
      + \frac{1}{2} \l \|   \bbnu_h^\frac{1}{4}  \l( \GRAD \bu_h + (\GRAD \bu_h)^\top \r) \bbnu_h^\frac{1}{4}  \r \|^2  \\ + \frac{1}{2} \l \|  \bbnu_{h,vms}^\frac{1}{4}  \l( \GRAD \bu_h + (\GRAD \bu_h)^\top - \proj \GRAD \bu_h - ( \proj \GRAD \bu_h)^\top \r) \bbnu_{h,vms}^\frac{1}{4} \r \|^2 = \l(- \frac{g \delta \rho_h}{\rho_0} \Hat{\by}, \bu_h \r)  - \frac{1}{2} \l(  \frac{\delta \rho_h  }{\rho_0} g (y - \overline{y})  , \DIV \bu_h \r),  
    \end{split}
  \end{equation}
  where we have also used that $\GRAD \overline{y} = 0$ since $\overline{y}$ is a constant.

The gravity term can, using integration by parts and $\hat{\by} = \GRAD (y - \overline{y})$, be written as 

  \begin{equation} \label{eq:gravity_step0}
    \begin{split}
      \l(- \frac{g \delta \rho_h}{\rho_0} \Hat{\by}, \bu_h \r)  - \frac{1}{2} \l(  \frac{\delta \rho_h }{\rho_0} g (y - \overline{y})  , \DIV \bu_h \r)   &= - \l( g \frac{\delta \rho_h}{\rho_0} \hat{\by} , \bu_h \r)  - \frac{1}{2} \l(  \frac{\delta \rho_h  }{\rho_0} g (y - \overline{y})  , \DIV \bu_h \r)         \\ &=  - \l(  \GRAD (y - \overline{y}) , \bu_h g \frac{\delta \rho_h}{\rho_0} \r)  - \frac{1}{2} \l(  \frac{\delta \rho_h  }{\rho_0} g (y - \overline{y})  , \DIV \bu_h \r)     \\ 
      &=    \l( \frac{g (y - \overline{y})}{\rho_0}  , \DIV  \l( \bu_h \delta \rho_h \r)   \r)  - \frac{1}{2} \l(  \frac{\delta \rho_h  }{\rho_0} g (y - \overline{y})  , \DIV \bu_h \r) 
      \\ &=   \l( \frac{g (y - \overline{y})}{\rho_0}  ,  \bu_h \SCAL \GRAD \delta \rho_h + (\DIV \bu_h) \delta \rho_h    \r)  - \frac{1}{2} \l(  \frac{  g (y - \overline{y}) }{\rho_0}  , (\DIV \bu_h ) \delta \rho_h \r) \\
      \\ &=   \l( \frac{g (y - \overline{y})}{\rho_0}  ,  \bu_h \SCAL \GRAD \delta \rho_h + \frac{1}{2} (\DIV \bu_h) \delta \rho_h    \r).
    \end{split}
  \end{equation}
Since $y - \overline{y}$ is a piecewise linear function with zero average $ (y - \overline{y}) C \in Q_{h,0} $ for any constant $C$ this means that $(C (y - \overline{y}) , \DIV \bu_h) = 0$. Therefore, \eqref{eq:gravity_step0} is equivalent to
      \begin{equation} \label{eq:gravity_step1}
\begin{split}
  &  \l(- \frac{g \delta \rho_h}{\rho_0} \Hat{\by}, \bu_h \r)  - \frac{1}{2} \l(  \frac{\delta \rho_h }{\rho_0} g (y - \overline{y})  , \DIV \bu_h \r) \\
      & =\l( \frac{g (y - \overline{y})}{\rho_0}  ,  \bu_h \SCAL \GRAD \delta \rho_h + \frac{1}{2} (\DIV \bu_h) (\delta \rho_h - \overline{\delta \rho_h})    \r) \\
     &= \l( g(y - \overline{y}), - \alpha_T ( \bu_h \SCAL \GRAD T_h + \frac{1}{2} (\DIV \bu_h) (T_h - \overline{T_h}) )  + \beta_S ( \bu_h \SCAL \GRAD S_h + \frac{1}{2} (\DIV \bu_h) (S_h - \overline{S_h}) ) \r), \end{split}
  \end{equation}
because $\overline{\delta \rho_h}$ is a constant.
Next, by setting $\w = g(y - \overline{y})$ in the tracer update \eqref{eq:tracer_stabilized_symmetric} for salinity and temperature, we further rewrite \eqref{eq:gravity_step1} as
\begin{equation} \label{eq:gravity_step2}
\begin{split}
    \l(- \frac{g \delta \rho_h}{\rho_0} \Hat{\by}, \bu_h \r)  - \frac{1}{2} \l(  \frac{\delta \rho_h }{\rho_0} g (y - \overline{y})  , \DIV \bu_h \r) \\
=- ( g (y - \overline{y})    ,  -\alpha_T  \p_t T_h   + \beta_S \p_t S_h ) - ( \GRAD y , -  g \alpha_T  (\kappa_T \polI + \bbkappa_{T,h} )  \GRAD T_h + g \beta_S ( \kappa_S \polI + \bbkappa_{S,h} )  \GRAD S_h    )  \\ - ( \GRAD y - \proj \GRAD y , - g \alpha_T   \bbkappa_{T,h,vms}  (\GRAD T_h - \proj \GRAD T_h) + g \beta_S   \bbkappa_{S,h,vms}   (\GRAD S_h - \proj \GRAD S_h)    ) \\
= -( g (y - \overline{y})    ,  -\alpha_T  \p_t T_h   + \beta_S \p_t S_h ) - ( \GRAD y , - g \alpha_T  (\kappa_T \polI + \bbkappa_{T,h} )  \GRAD T_h + g \beta_S ( \kappa_S \polI + \bbkappa_{S,h} )  \GRAD S_h    ),
\end{split}
  \end{equation}
where $\proj \GRAD y = \GRAD y$ was used in the last step. This is true because the projection of a constant is the same as the constant. Finally, by applying integration by parts repeatedly, we have that
      \begin{equation} \label{eq:last_step}
        \begin{split}
          &  \l(- \frac{g \delta \rho_h}{\rho_0} \Hat{\by}, \bu_h \r)  - \frac{1}{2} \l(  \frac{\delta \rho_h }{\rho_0} g (y - \overline{y})  , \DIV \bu_h \r) \\
      & = -( g (y - \overline{y})    ,  -\alpha_T  \p_t T_h   + \beta_S \p_t S_h ) - ( \hat{\by} , -  g\alpha_T  (\kappa_T \polI + \bbkappa_{T,h} )  \GRAD T_h + g\beta_S ( \kappa_S \polI + \bbkappa_{S,h} )  \GRAD S_h    ) \\
      & = -( g (y - \overline{y})    ,  -\alpha_T  \p_t T_h   + \beta_S \p_t S_h ) + (g \alpha_T (\kappa_T \polI + \bbkappa_{T,h} ) \hat{\by} ,      \GRAD T_h) - ( g\beta_S ( \kappa_S \polI + \bbkappa_{S,h} )  \hat{\by},   \GRAD S_h    ) \\
      & = -( g (y - \overline{y})    ,  -\alpha_T  \p_t T_h   + \beta_S \p_t S_h ) - ( \DIV(  (\kappa_T \polI + \bbkappa_{T,h} )  \hat{\by} ) ,    g \alpha_T  T_h) + ( \DIV (  ( \kappa_S \polI + \bbkappa_{S,h} ) \hat{\by} ),  g \beta_S S_h    ) \\ & +  ( \bn \SCAL ( (\kappa_T \polI + \bbkappa_{T,h})\hat{\by} ), g \alpha_T T_h )_{\p \Omega} - ( \bn \SCAL ( (\kappa_S \polI + \bbkappa_{S,h})\hat{\by} ), g \beta_S S_h )_{\p \Omega}. \\
    \end{split}
  \end{equation}
Inserting this into \eqref{eq:proof_step0} concludes the proof.


\end{proof}

\begin{theorem}
  The stabilized FEM conserves angular momentum in the sense that
  \begin{equation}
   \p_t \int_\Omega (\bu_{h} \CROSS \bx)_i \ud \bx = - \int_\Omega \l( \frac{g \delta \rho_h}{\rho_0} \Hat{\by} \CROSS \bx \r)_i \ud \bx. 
  \end{equation}
\end{theorem}

\begin{proof}
We define $\bpsi_i \coloneqq  \bx \CROSS \be_i $ and note that $\bpsi_i$ has the properties $\DIV \bpsi_i = 0$ and $ \GRAD \bpsi_i + (\GRAD \bpsi_i)^\top = 0 $. For a correct definition of the cross product in 2D, the last component of all vectors is extended by 0. Following \citep{Charnyi2017,Ingimarson2023}, let $\Omega = \hat{\Omega} \cup \Omega_s$ where $\hat{\Omega}$ is a strictly interior subdomain of $\Omega$. We define the restriction $\chi(\bg) \in \bcalV_h$ of an arbitrary function $\bg$ by setting $\chi(\bg) = \bg$ in $\hat{\Omega}$ and $\chi(\bg) = 0$ in $\Omega_s$. We assume that $\bv$ and $\bu_h$ are zero in $\Omega_s$. We set $\bv = \chi( \bpsi_i)$ and obtain
  \begin{equation} \label{eq:ang_mom_proof}
    \begin{split}
    \p_t \int_\Omega (\bu_{h} \CROSS \bx)_i \ud \bx = & ( \p_t \bu_h , \bpsi_i) = - b(\bu_h,\bu_h,\bpsi_i) - b(\bpsi_i, \bu_h, \bu_h) - (  (\DIV \bu_h) \bu_h , \bpsi_i ) + ( P_h, \DIV \bpsi_i) \\ &- (\gamma_h \DIV \bu_h, \DIV \bpsi_i) - \frac{1}{2} \Biggl (   \nu    \l( \GRAD \bu_h + (\GRAD \bu_h)^\top \r) \Biggr . \\  & \Biggl . +   \sqrt{\bbnu_h}  \l( \GRAD \bu_h + (\GRAD \bu_h)^\top \r)  \sqrt{\bbnu_h}, \GRAD \bpsi_i + \GRAD \bpsi_i^\top \Biggr )   \\ &  - \frac{1}{2} \Biggl ( \sqrt{ \bbnu_{h,vms} }  \l( \GRAD \bu_h + (\GRAD \bu_h)^\top - \proj \GRAD \bu_h - (\proj \GRAD \bu_h)^\top  \r)  \sqrt{ \bbnu_{h,vms} } , \Biggr . \\
     & \Biggl . \GRAD \bpsi_i + \GRAD \bpsi_i^\top - \proj \GRAD \bpsi_i - \proj \GRAD \bpsi_i^\top \Biggr )    \\
      &  - \l(   \frac{g \delta \rho_h}{\rho_0} \Hat{\by} , \bpsi_i \r)  -  \frac{1}{2} \l(   \frac{\delta \rho_h  }{\rho_0} g y  , \DIV \bpsi_i \r) \\
      = &   b(\bu_h,\bpsi_i,\bu_h)  - \l(   \frac{g \delta \rho_h}{\rho_0} \Hat{\by} , \bpsi_i \r) = b(\bu_h,\bpsi_i,\bu_h) - \int_\Omega \l( \frac{g \delta \rho_h}{\rho_0} \Hat{\by} \CROSS \bx \r)_i \ud \bx, 
    \end{split}%
  \end{equation}
where $\DIV \bpsi_i = 0$, $\GRAD \bpsi_i + \GRAD \bpsi_i^\top = 0$, $\proj \GRAD \bpsi_i = \GRAD \bpsi_i$, and \eqref{eq:IBP1} and \eqref{eq:IBP2} was used.

Next, by using the definition of the trilinear form \eqref{eq:trilinear definition} it can be shown that
\begin{equation} \label{eq:trillinear angular momentum}
\begin{split}
 b(\bu_h,\bpsi_i, \bu_h) & = \frac{1}{2}b(\bu_h, \bpsi_i, \bu_h) + \frac{1}{2}b(\bu_h, \bpsi_i, \bu_h )   = \frac{1}{2} \l( (\GRAD \bpsi_i)^\top \bu_h, \bu_h  \r) + \frac{1}{2} \l( (\GRAD \bpsi_i)^\top \bu_h, \bu_h  \r) \\ & = \frac{1}{2} \l( (\GRAD \bpsi_i)^\top \bu_h, \bu_h  \r) + \frac{1}{2} ( (\GRAD \bpsi_i) \bu_h, \bu_h  ) = 0,
\end{split}
\end{equation}
since $\GRAD \bpsi_i + (\GRAD \bpsi_i)^\top = 0$. Inserting this into \eqref{eq:ang_mom_proof} concludes the proof.
\end{proof}

\begin{proposition}
  The stabilized tracer update \eqref{eq:tracer_stabilized_symmetric} is shift-invariant \eqref{eq:shift-invariant}, conserves tracer mass $\int_\Omega \phi_h \ud \bx$ and dissipates tracer energy $\int_\Omega \phi^2_h \ud \bx$.
\end{proposition}
\begin{proof}
  Mass conservation can be shown be setting $w = 1$ and using that $\l( \DIV \bu_h, \phi_h - \overline{\phi_h} \r) \in Q_{h,0}$. Tracer energy dissipation is shown by setting $w = \phi$. Shift-invariance is straightforwardly shown from \eqref{eq:shift-invariant}. See \citep{lundgren2024fully} for more details.
\end{proof}


\section{Numerical experiments}\label{sec:numexp}
We run two sets of experiments to validate our numerical method. The first set of experiments are aimed at validating the new theoretical results and the accuracy of the tensor-based stabilization. The purpose of the second experiment set is to test the method in a more realistic setting relevant to ocean circulation and glacial melting. We do this by reproducing one of the idealized Sherard Osborn fjord high-resolution MITgcm simulations from \citep{Wiskandt}. MITgcm has a numerical FV kernel for atmospheric and oceanic hydro-thermodynamic flow and has been used in modeling studies relevant to ice-sheet-ocean interactions in Greenland and Antarctica \citep{sciascia_seasonal_2013,cai_observations_2017,mitgcm_GOLDBERG,mitgcm_jordan}. 


Our FEM is implemented in the open-source finite element library FEniCS \citep{Alnaes2015,fenics:book} with PETSc \citep{webpage:petsc,BalayEtAl1997} and hypre \citep{hypre} as linear algebra backends. In time we use high-order BDF time-stepping which is efficient and easy to implement. In particular, we use variable timestep BDF4 time-stepping \citep{Wang2008} together with a simple adaptive time-stepping routine following a given user-defined CFL-number. The nonlinear problem is linearized using the so-called Newton linearization from \citep{Charnyi2019}. We use the saddle-point preconditioning technique from \citep{lundgren_2024RV}, which employs a Schur complement approximation inspired by artificial compressibility methods \citep{Lundgren2022, Chorin1967, Guermond_Minev_2015}. The saddle-point technique is suitable for high Reynolds number flow. Additionally, an extrapolation technique is used to generate an initial guess, which speeds up the convergence of the iterative solver \citep{lundgren_2024RV}. Unless otherwise stated, we choose the following numerical parameters for all experiments: $C_{\text{max}} = 1.0, C_{\text{max,vms}} = 0.05, C_\Delta = 10$, $C_{\text{flat}} = 0.1$ $f_{RV} = 15 x^2$. We also let $(\phi_h, \bu_h, p_h)\in\polP_2 \CROSS \polP_3 \CROSS \polP_2$. 

\subsection{Numerical validation}

\subsubsection{No-flow problem} \label{Sec:noflow}
The primary purpose of this experiment is to numerically verify Theorem \ref{proposition_potential} and to test how well the method satisfies hydrostatic balance. Another goal is to compare the accuracy of the new SI-MEEDMAC formulation \eqref{eq:ns_EMAC_potential} to the SI-MEDMAC formulation \eqref{eq:ns_EMAC}. We note that SI-MEDMAC corresponds to EMAC \citep{Charnyi2017} when density is constant. We simplify the gravity force $\bef = - \frac{g \delta \rho}{\rho_0} \Hat{\by} = T \hat{\by}$ and consider the following no-flow solution to the Boussinesq system
\begin{equation}
  \begin{split}
    T &= \frac{1}{2} \tanh(5y) + 10, \\ 
    \bu &= \begin{bmatrix}
      0 \\ 0
    \end{bmatrix}, \\
    p &=  \frac{1}{10} \ln(\cosh(5y)) + 10 y - \frac{1}{|\Omega|} \int_\Omega \frac{1}{10} \ln(\cosh(5y)) + 10 y d \bx, \\ 
    P &=  \frac{1}{10} \ln(\cosh(5y)) + 10 y - \frac{1}{2} y T - \frac{1}{|\Omega|} \int_\Omega \frac{1}{10} \ln(\cosh(5y)) + 10 y - \frac{1}{2} y T d \bx, \\ 
  \end{split}
  \end{equation}
on the square $\Omega = [-1,1]\CROSS [-1,1]$ with no-slip boundary conditions for velocity and homogeneous Neumann boundary conditions for temperature. We perform a convergence study on a series of unstructured meshes with the constant time-step $\Delta t = h$ and with $\nu = 0.01$ and $\kappa = 0$. We also set the stabilization to zero ($C_{\text{max}} = C_{\text{max,vms}} = 0$). The convergence results are presented in Table \ref{table:noflow} and we observe the expected convergence rates for all components for both formulations. The errors are not normalized. We note that the new formulation is more accurate.

Next, we set $\kappa = \nu = 0$. Since there is no dissipation present, Theorem \ref{proposition_potential} then predicts that the sum of kinetic energy and potential energy $\| \bu_h \| - (T, y)$ should be conserved. Since time-discretization also affects this balance, for this experiment, we use the standard Crank-Nicolson method which is energy conservative \citep{Charnyi2017,lundgren2024fully}. We perform the experiment again with $4909$ $\polP_3$ and $2209$ $\polP_2$ nodes. In \fgref{fig:potential} we observe that total energy is conserved using our new formulation \eqref{eq:ns_EMAC_potential} but not with EMAC \eqref{eq:ns_EMAC}. We also observe that the new formulation is orders of magnitudes more accurate because of this.


\begin{table}[H] 
  \centering     
  \caption{No-flow problem on an unstructured mesh. Convergence study with end time $\widehat{t}=1$, $ \nu=0.01$, $\kappa = 0$, $\Delta t = h$, $(T_h, \bu_h, p_h)\in\polP_2 \CROSS \polP_3 \CROSS \polP_2$ with no stabilization. We compare the SI-MEDMAC formulation \eqref{eq:ns_EMAC} with the SI-MEEDMAC formulation \eqref{eq:ns_EMAC_potential}.} \label{table:noflow}%
  \begin{tabular}{|c|c|c|c|c|c|c|c|}    
    \multicolumn{8}{c}{Velocity}\\
    \cline{1-8}
  &\multicolumn{1}{c|}{$\#$ DOFs} & \multicolumn{1}{c|}{$L_1$} & \multicolumn{1}{c|}{rate} & \multicolumn{1}{c|}{$L_2$} & \multicolumn{1}{c|}{rate} & \multicolumn{1}{c|}{$L_{\infty}$} & \multicolumn{1}{c|}{rate}\\
   \hline
  \parbox[t]{2mm}{\multirow{5}{*}{\rotatebox[origin=c]{90}{\footnotesize{SI-MEDMAC}}}} & 
         1754 &   7.40E-04 &   -- &   4.78E-04 &   -- &   1.61E-03 &   -- \\ 
  &      6314 &   4.76E-05 &   4.28 &   3.13E-05 &   4.26 &   1.86E-04 &   3.37 \\ 
  &      9818 &   1.85E-05 &   4.28 &   1.27E-05 &   4.10 &   8.06E-05 &   3.79 \\ 
  &     24164 &   3.22E-06 &   3.88 &   2.22E-06 &   3.87 &   1.41E-05 &   3.88 \\ 
  &     97232 &   1.96E-07 &   4.02 &   1.38E-07 &   3.99 &   1.07E-06 &   3.70 \\  \hline
  \parbox[t]{2mm}{\multirow{5}{*}{\rotatebox[origin=c]{90}{\footnotesize{SI-MEEDMAC}}}} & 
        1754 &   6.03E-04 &   -- &   4.17E-04 &   -- &   9.36E-04 &   -- \\ 
  &      6314 &   3.11E-05 &   4.63 &   2.01E-05 &   4.73 &   6.88E-05 &   4.08 \\ 
  &      9818 &   1.18E-05 &   4.38 &   8.14E-06 &   4.10 &   3.45E-05 &   3.12 \\ 
  &     24164 &   1.90E-06 &   4.06 &   1.26E-06 &   4.15 &   7.97E-06 &   3.26 \\ 
  &     97232 &   1.13E-07 &   4.06 &   7.53E-08 &   4.04 &   5.48E-07 &   3.85 \\ \hline 
    \multicolumn{8}{c}{Temperature}\\
    \cline{1-8}
    \hline
  &\multicolumn{1}{c|}{$\#$ DOFs} & \multicolumn{1}{c|}{$L_1$} & \multicolumn{1}{c|}{rate} & \multicolumn{1}{c|}{$L_2$} & \multicolumn{1}{c|}{rate} & \multicolumn{1}{c|}{$L_{\infty}$} & \multicolumn{1}{c|}{rate}\\
   \hline
  \parbox[t]{2mm}{\multirow{5}{*}{\rotatebox[origin=c]{90}{\footnotesize{SI-MEDMAC}}}} & 
         401 &   2.51E-03 &   -- &   3.57E-03 &   -- &   2.01E-02 &   -- \\ 
  &      1425 &   4.01E-04 &   2.89 &   5.49E-04 &   2.95 &   2.25E-03 &   3.45 \\ 
  &      2209 &   1.68E-04 &   3.98 &   2.37E-04 &   3.84 &   1.88E-03 &   0.83 \\ 
  &      5413 &   4.51E-05 &   2.93 &   6.10E-05 &   3.03 &   3.93E-04 &   3.49 \\ 
  &     21693 &   5.43E-06 &   3.05 &   7.69E-06 &   2.98 &   6.49E-05 &   2.60 \\   \hline
  \parbox[t]{2mm}{\multirow{5}{*}{\rotatebox[origin=c]{90}{\footnotesize{SI-MEEDMAC}}}} & 
        401 &   2.51E-03 &   -- &   3.53E-03 &   -- &   1.98E-02 &   -- \\ 
  &      1425 &   3.97E-04 &   2.91 &   5.40E-04 &   2.96 &   2.20E-03 &   3.46 \\ 
  &      2209 &   1.66E-04 &   3.97 &   2.34E-04 &   3.82 &   1.86E-03 &   0.78 \\ 
  &      5413 &   4.48E-05 &   2.92 &   6.05E-05 &   3.02 &   3.90E-04 &   3.49 \\ 
  &     21693 &   5.42E-06 &   3.04 &   7.67E-06 &   2.97 &   6.47E-05 &   2.59 \\ \hline
  
    \multicolumn{8}{c}{Pressure}\\
    \cline{1-8}
  
    \hline
  &\multicolumn{1}{c|}{$\#$ DOFs} & \multicolumn{1}{c|}{$L_1$} & \multicolumn{1}{c|}{rate} & \multicolumn{1}{c|}{$L_2$} & \multicolumn{1}{c|}{rate} & \multicolumn{1}{c|}{$L_{\infty}$} & \multicolumn{1}{c|}{rate}\\
   \hline
  \parbox[t]{2mm}{\multirow{5}{*}{\rotatebox[origin=c]{90}{\footnotesize{SI-MEDMAC}}}}  &
         401 &   3.97E-04 &   -- &   3.92E-04 &   -- &   1.49E-03 &   -- \\ 
  &      1425 &   5.12E-05 &   3.23 &   5.39E-05 &   3.13 &   2.05E-04 &   3.13 \\ 
  &      2209 &   2.20E-05 &   3.85 &   2.48E-05 &   3.54 &   9.11E-05 &   3.70 \\ 
  &      5413 &   5.79E-06 &   2.98 &   6.95E-06 &   2.84 &   3.28E-05 &   2.28 \\ 
  &     21693 &   7.38E-07 &   2.97 &   8.59E-07 &   3.01 &   5.03E-06 &   2.70 \\    \hline
  \parbox[t]{2mm}{\multirow{5}{*}{\rotatebox[origin=c]{90}{\footnotesize{SI-MEEDMAC}}}} & 
         401 &   3.17E-04 &   -- &   3.21E-04 &   -- &   1.17E-03 &   -- \\ 
  &      1425 &   3.59E-05 &   3.44 &   3.84E-05 &   3.35 &   1.61E-04 &   3.13 \\ 
  &      2209 &   1.54E-05 &   3.86 &   1.59E-05 &   4.01 &   6.80E-05 &   3.94 \\ 
  &      5413 &   3.95E-06 &   3.03 &   4.24E-06 &   2.95 &   2.45E-05 &   2.28 \\ 
  &     21693 &   4.79E-07 &   3.04 &   5.19E-07 &   3.03 &   3.55E-06 &   2.78 \\  \hline
  \end{tabular}                                                                   
  \label{table:ms_P3P3P2}                                                      
  \end{table}

  \begin{figure}[H]
    \centering 
    \begin{subfigure}[position]{0.49 \textwidth}
      \includegraphics[width=1.0\textwidth]{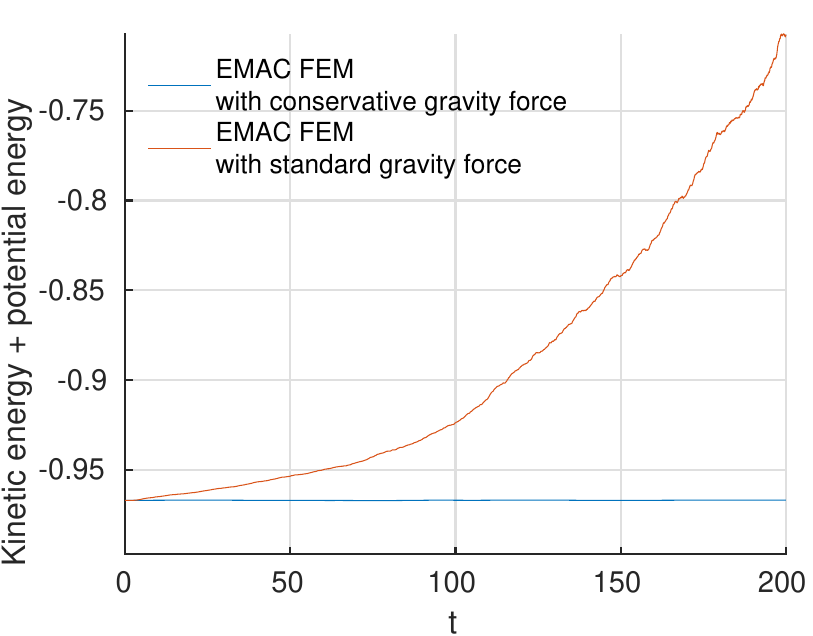}
      \caption{}
    \end{subfigure}     
    \begin{subfigure}[position]{0.49 \textwidth}
      \includegraphics[width=1.0\textwidth]{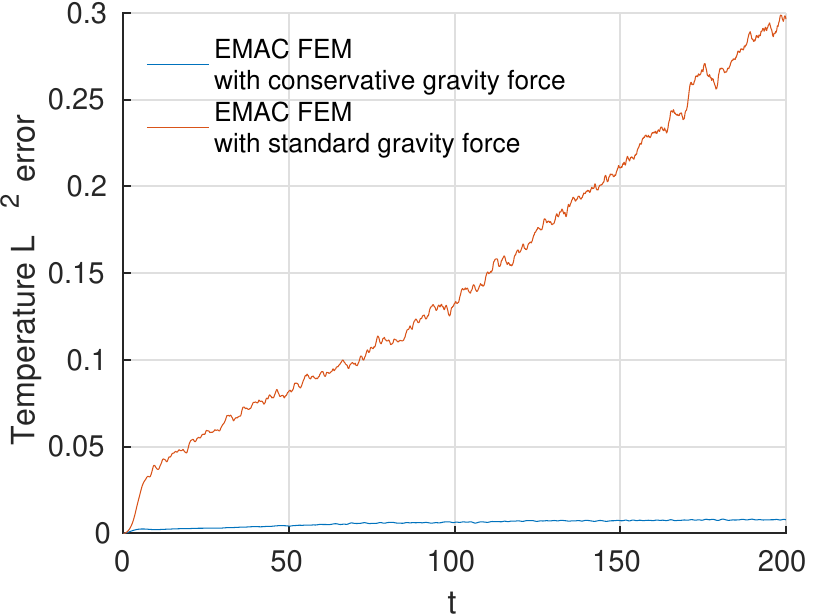}
      \caption{}
    \end{subfigure}
    \caption{Time evolution of (a) kinetic + potential energy and (b) $L^2$-error for $T_h$ using the SI-MEDMAC formulation \eqref{eq:ns_EMAC} (red) and the SI-MEEDMAC formulation \eqref{eq:ns_EMAC_potential} (blue). No stabilization is used and $\nu = \kappa = 0$. The unstructured mesh consists of $4909$ $\polP_3$ and $2209$ $\polP_2$ nodes.}    
    \label{fig:potential}
  \end{figure}

\subsubsection{Accuracy of the tracer advection solution}
In this section, we validate the accuracy of our FEM against the scalar advection equation on the square $\Omega = [-1,1]\CROSS [-1,1]$. We prescribe $\bu = \begin{bmatrix}
  -2 \pi y & 2 \pi x
\end{bmatrix}^\top$ and consider the following smooth tracer solution
\begin{equation}
  \phi = 2 + \frac{1}{2} \l( 1 - \tanh \l( \frac{ (x - x_0)^2 + (y - y_0)^2 }{r_0^2}  -1 \r) \r),
\end{equation}
with $ x_0 = 0.35 \cos(2 \pi t) $, $y_0 = 0.35 \sin( 2 \pi t )$ and $r_0 = 0.25$. We solve the problem until $\widehat{t} = 1$, which equals the overturning time-scale, \ie we solve for a full rotation and vary the element order between one and four. We set $CFL = 0.15$. The results are presented in Table \ref{table:convergence_advection_P1} and show suboptimal convergence for the unstabilized Galerkin method and nearly optimal convergence for the stabilized scheme. This is mainly due to the high-order dissipation present in the stabilized method and similar results were reported by \cite[Table 1]{kuzmin2023_weno} indicating that the high-order dissipation improves accuracy for smooth solutions. Since the solution is smooth the resulting residual is small and the low-order diffusion does not negatively affect the accuracy of the scheme.


                  \begin{table}[H]
                    \centering
                    \caption{Advection equation in 2D. Convergence study for $\phi_h$ for different polynomial degrees at time $\widehat{t}=1$ on a series of unstructured meshes. The errors are normalized with their corresponding norm.}
                    \label{table:convergence_advection_P1}
                \vspace{0.in}
                $\polP_1$ elements
                \\
                \vspace{0.1in}
                
                    \begin{adjustbox}{max width=\textwidth}
                    \begin{tabular}{c|c|c|c|c|c|c|c|c}
                    \hline
                    \multirow{2}{*}{\#DOFs} &  \multicolumn{4}{c|}{Galerkin} &  \multicolumn{4}{c}{Residual viscosity}  \\ \cline{2-9}
                     {}   &       L$^1$   &    Rate   &       L$^2$   &    Rate &       L$^1$   &    Rate   &       L$^2$   &    Rate   \\ \hline
                     5359 &   9.13E-04 &   -- &   1.25E-03 &   -- &   1.83E-04 &   -- &   5.29E-04 &   -- \\
                     21677 &   2.95E-04 &   1.62 &   4.15E-04 &   1.58 &   3.53E-05 &   2.35 &   1.03E-04 &   2.34 \\
                     86629 &   9.46E-05 &   1.64 &   1.39E-04 &   1.58 &   7.93E-06 &   2.15 &   2.35E-05 &   2.13 \\
                    346222 &   3.07E-05 &   1.62 &   4.84E-05 &   1.53 &   1.92E-06 &   2.05 &   5.63E-06 &   2.06  \\
                   \hline
                    \end{tabular}
                    \end{adjustbox}
                \\
                \vspace{0.2in}
                $\polP_2$ elements
                \\
                \vspace{0.1in}
                
                    \begin{adjustbox}{max width=\textwidth}
                    \begin{tabular}{c|c|c|c|c|c|c|c|c}
                    \hline
                    \multirow{2}{*}{\#DOFs} &  \multicolumn{4}{c|}{Galerkin} &  \multicolumn{4}{c}{Residual viscosity}  \\ \cline{2-9}
                     {}   &       L$^1$   &    Rate   &       L$^2$   &    Rate &       L$^1$   &    Rate   &       L$^2$   &    Rate   \\ \hline
                           5413 &   5.22E-04 &   --   &   6.99E-04 &   --   &   6.04E-05 &   --   &   1.64E-04 &   --   \\
                          21693 &   1.12E-04 &   2.22 &   1.53E-04 &   2.19 &   8.00E-06 &   2.91 &   2.34E-05 &   2.81 \\
                          87221 &   2.07E-05 &   2.42 &   2.87E-05 &   2.40 &   1.19E-06 &   2.74 &   3.68E-06 &   2.66 \\
                         347541 &   3.55E-06 &   2.55 &   5.28E-06 &   2.45 &   1.71E-07 &   2.81 &   5.58E-07 &   2.73 \\
                   \hline
                    \end{tabular}
                    \end{adjustbox}
                
                \vspace{0.2in}
                $\polP_3$ elements
                \\
                \vspace{0.1in}
                
                   \label{table:convergence_advection_P3}
                    \begin{adjustbox}{max width=\textwidth}
                    \begin{tabular}{c|c|c|c|c|c|c|c|c}
                    \hline
                    \multirow{2}{*}{\#DOFs} &  \multicolumn{4}{c|}{Galerkin} &  \multicolumn{4}{c}{Residual viscosity}   \\ \cline{2-9}
                     {}   &       L$^1$   &    Rate   &       L$^2$   &    Rate &       L$^1$   &    Rate   &       L$^2$   &    Rate   \\ \hline
                     4909 &   2.42E-04 &   -- &   3.80E-04 &   -- & 4.60E-05 &   -- &   1.25E-04 &   -- \\
                   12082 &   4.16E-05 &   3.92 &   6.56E-05 &   3.90 & 4.71E-06 &   5.06 &   1.52E-05 &   4.68 \\
                   48616 &   3.42E-06 &   3.59 &   5.38E-06 &   3.59 & 1.96E-07 &   4.56 &   6.80E-07 &   4.47 \\
                  195862 &   3.03E-07 &   3.48 &   5.10E-07 &   3.38 & 1.05E-08 &   4.21 &   3.88E-08 &   4.11 \\
                   \hline
                    \end{tabular}
                    \end{adjustbox}
                    \\
                    \vspace{0.2in}
                    $\polP_4$ elements%
                    \\
                    \vspace{0.1in}
                    \label{table:convergence_advection_P4}
                    \begin{adjustbox}{max width=\textwidth}
                    \begin{tabular}{c|c|c|c|c|c|c|c|c}
                    \hline
                    \multirow{2}{*}{\#DOFs} &  \multicolumn{4}{c|}{Galerkin} &  \multicolumn{4}{c}{Residual viscosity}   \\ \cline{2-9}
                     {}   &       L$^1$   &    Rate   &       L$^2$   &    Rate &       L$^1$   &    Rate   &       L$^2$   &    Rate   \\ \hline
                            8673 &   2.50E-05 &   -- &   3.76E-05 &   -- &   3.34E-06 &   -- &   9.65E-06 &   -- \\ 
                          21393 &   3.75E-06 &   4.21 &   5.83E-06 &   4.13 &   3.25E-07 &   5.16 &   1.10E-06 &   4.80 \\ 
                          86257 &   1.93E-07 &   4.25 &   3.07E-07 &   4.22 &   1.10E-08 &   4.86 &   4.07E-08 &   4.73 \\ 
                         347857 &   8.56E-09 &   4.47 &   1.45E-08 &   4.38 &   3.76E-10 &   4.84 &   1.45E-09 &   4.79\\
                   \hline
                    \end{tabular}
                    \end{adjustbox}
                \end{table}

\subsection{The Sherard Osborn fjord}\label{sec:ryder}
We now simulate a 2D representation of the Sherard Osborn fjord, consistent with the MITgcm model of \citep{Wiskandt}, see \fgref{fig:CirculationMelt}. The domain is $32$ \si{kilometers} long and $1000$ \si{meters} deep with an ice-tongue geometry that is an idealization based on observed data of Ryder glacier-Sherard Osborn system \citep{Jakobsson_etal2020}. The ice tongue is $20$ \si{kilometers} long and has a $50$ \si{meters} thick vertical front. The base of the tongue has a linear slope of $0.045$ \si{degrees}, which extends to the grounding line where there is $50$ \si{meters} high vertical wall.
The boundary and initial conditions are described in Section \ref{sec:bcs}. The width of the restoring region $x_r$ is 2 km, and $x_{oc}=32$ km. We simulate for 35 days, as this is more than twice the overturning time scale of the problem. The time step is given by the CFL condition, with a CFL constant set to $0.25$. The relative tolerance of the iterative solver is set to $10^{-6}$.

We reproduce Figure 1a (melt rates, density, streamlines), 1d (temperature and salinity profiles), 3a (horizontal velocity) and B1 (evolution of kinetic energy, average melt rate and temperature change) of \citep{Wiskandt}.
In \citep{Wiskandt}, the Atlantic water temperature at the open ocean boundary as well as the subglacial discharge rate is varied in order to perform sensitivity studies. As the aim of this paper is to compare the results of our continuous Galerkin FEM model to FV models like MITgcm we only run the control experiment representing winter conditions from \citep{Wiskandt}, i.e., we use a low-temperature boundary condition ($T_{AW}$ below) at the open ocean and do not inject subglacial discharge. This experiment is denoted \verb|control_win| in \citep{Wiskandt}.

\subsubsection{Restoring function}
The profile $\phi_{res}$ which serves as boundary and initial conditions for temperature and salinity is constructed from measured data just in front of the ice front (Station 16, 17 from Fig. 1 in \citep{Jakobsson_etal2020}), as in \citep{Wiskandt}. 

%
%
\begin{equation}\label{eq:phires}
  \phi_{res} = \frac{(\phi_{AW}+\phi_{PW})}{2}+ \frac{\phi_{AW}-\phi_{PW}}{2}\tanh{\left(5\pi \frac{d_{pycno}-y}{d_{total}}  \right)},
\end{equation}
where $d_{pycno}=800$ is the distance from the ocean floor to the pycnocline, $d_{total}=1000$ is the total depth of the ocean, $\phi_{AW}=T_{AW}=0.2$, $\phi_{PW}=T_{PW}=-1.6$ (\si{\celsius}) if $\phi$ represents temperature $T$, and $\phi_{AW}=S_{AW}=35 $,  $\phi_{PW}=S_{AW}=34$ (\psu) if $\phi$ represents salinity. 
The profile \Eqref{eq:phires} reflects that the Sherard Osborn is stratified with warm and saline Atlantic water at depth and a colder and fresher polar water of Arctic origin layer closer to the surface, as is typical for a Greenlandic fjord \citep{Straneo_etal2012}.%
\begin{table}
  \caption{Parameters related to density (the equation of state), viscosity, diffusivity and melt parameterization.}
  \centering
  \begin{tabular}[h!]{l l l l}
    \hline
    \hline
    Symbol & Value & Units & Parameter\\
    \hline
    $\kappa_{T}$ & $1.41 \times 10^{-7}$ & \si{\meter^2 \per \second^2}& Thermal diffusivity of seawater \\
       $\kappa_{S}$ & $8.01 \times 10^{-10}$ & \si{\meter^2 \per \second^2}& Salinity diffusivity of seawater \\ 
    $h_{\epsilon}$ & $3 + \frac{1}{3} $ & \si{\meter}& Boundary layer thickness\\
        $\nu$ & $1.95\times 10^{-6}$ & \si{\meter\squared\per\second} & Kinematic viscosity of seawater\\
    $c_d$ & $1.5\times 10^{-3}$ & --- & Drag coefficient\\
    $K$ & 0.40 & --- & Von K\'arm\'an's constant\\
    $c_i$ & 2009.0 & \si{\joule\per\kilogram\per\kelvin} & Specific heat capacity ice shelf\\
    $c_m$ & 3974.0 & \si{\joule\per\kilogram\per\kelvin} & Specific heat capacity mixed layer\\
    $T_0$ &0 & \si{\celsius}  & Reference water temperature\\
    $S_0$ &35 & \si{\psu} & Reference water salinity\\
    $\rho_0$ & 999.8 & \si{kg\per\cubic\meter} & Reference density\\
    
    $\alpha_T$ & $0.4\times 10^{-4}$ & \si{\per\celsius} & Thermal expansion coefficient\\
    $\beta_S$ & $8\times 10^{-4}$ & \si{\per\celsius} & Salinity expansion coefficient\\

    $\xi_N$ & 0.052 & --- & Stability constant \\
    $\Pi^{S}$ & 2432 & --- & Schmidt number (Sc)\\
    $\Pi^{T}$ & 13.8 & --- & Prandtl number (Pr)\\
    $\lambda_1$ & $3.34 \times 10^{5}$ & \si{\celsius\per\psu} & Salinity coefficient of freezing equation\\
    $\lambda_2$ & -0.0573 & --- & Constant in freezing equation\\
    $\lambda_3$ & 0.0939& \si{\celsius\per\pascal} & Pressure coefficient of freezing equation\\
    \hline
  \end{tabular}
  \label{tab:parameters}
\end{table}
\subsubsection{Melt rate parameterization and its discretization}

The tracer fluxes $F_{\phi}$ used to specify the boundary conditions \eqref{eq:tracer_bc_flux} at $\Gamma_{ice}$ are given by a virtual flux formulation \citep{JenkinsEtAl2001}
\begin{subequations}
  \label{eq:tracer_flux}
  \begin{align}
    F_S =& (\gamma_S + m_b)(S_m - S_b), \label{eq:salinity_flux}\\
    F_T =& (\gamma_T + m_b)(T_m - T_b), \label{eq:temp_flux}
  \end{align}
\end{subequations}
where the temperature and salinity at the base, $T_b$ and $S_b$, and the melt rate, $m_b$, are the solutions to the three-equation formulation \citep{HollandJenkins1999} which is solved node-wise at each time step:
\begin{subequations}
  \begin{align}
    T_b &= \lambda_1 S_b + \lambda_2 + \lambda_3 p_b\label{eq:3eqs_temp}\\
    m_b L + m_b c_i (T_b - T_i) &= c_m\gamma_T (T_m - T_b)\label{eq:3eqs_melt}\\
    m_b S_b &= \gamma_S (S_m - S_b).\label{eq:3eqs_salt}
  \end{align}
  \label{eq:3eqs}%
\end{subequations}
Here the subscripts $_m$ and $_i$ refer to the parameter value in the mixed ocean layer at the base of the ice shelf and of the ice shelf itself, respectively, while $c_{i,m}$ are the specific heat capacities and $\lambda_{1,2,3}$ are coefficients (see Table \ref{tab:parameters}). The tracer exchange velocity, $\gamma_{\phi}$ (where $\phi$ as above denotes either $T$ or $S$) is in here a compressed form of \cite{HollandJenkins1999}
\begin{equation}
  \label{eq:gamma_tracer}
  \gamma_{\phi} = \frac{\sqrt{c_d}u_m}{\Gamma_{Turb} + \Gamma^{\phi}_{Mole}}, \quad
    \Gamma_{Turb} =  \frac{1}{2\xi_N} - \frac{1}{K}, \quad   \Gamma^{\phi}_{Mole} = 12.5 \Pi^{\phi}. 
\end{equation}
In \eqref{eq:gamma_tracer}, the only variable is $u_m$, the velocity of the mixed layer. All other parameters are specified constants and are given in Table \ref{tab:parameters}.

Following \citet{Kimura_etal2013}, the values in the mixed layer ($T_m, S_m, u_m$) are evaluated at a distance of the boundary layer thickness $h_{\epsilon}$ away from the boundary in the normal direction. 

In this paper, we compute the (discrete) unit boundary normal $\bn_h$ as the area-weighted average of the normals on all cell (triangle) facets connecting to the DOF \citep[e.g.][]{John2002, Nazarov_Larcher_2017}. 
For a DOF located at coordinate $\bx_d$, we then compute mixed-layer tracer and velocity values as:
\begin{align*}
  u_m = \max\l( \bu_h(\bx + \bn_h h_{\epsilon}), 10^{-3} \r),
  T_m = T_h(\bx + \bn_h h_{\epsilon}),
  S_m = S_h(\bx + \bn_h h_{\epsilon}).
\end{align*}

Given the mixed layer values ($T_m, S_m, u_m$), the three-equation system \eqref{eq:3eqs} can then be reduced to a quadratic equation solving for $S_b$ (one solution of which will always be positive \citep[e.g.][]{JenkinsEtAl2001}) which is in turn used to solve for $T_b$ and $m_b$. Finally, to ensure that the resulting heat and salt fluxes in  \eqref{eq:tracer_flux} are smooth, we solve the following projection problem: Find $\projS F_\phi \in \calM_h$ such that
\begin{equation}
(F_{\phi,h} , w) = (F_\phi , w) + C_{\Delta} (\GRAD F_\phi , \GRAD w), \quad \forall w \in \calM_h,
\end{equation}
where $C_{\Delta}$ is a constant.



\subsubsection{Variable resolution, fully unstructured mesh}
The mesh is unstructured under the ice shelf, and structured north (right) of the ice shelf, see the left panel of \fgref{fig:rydermesh}. The unstructured mesh allows for representation of the ice-shelf geometry while a structured mesh outside the ice shelf allows for an accurate representation of the isopycnals in this area where the flow is closer to hydrostatic equilibrium. The mesh is refined close to the ice, and around the pycnocline in order to resolve the buoyant meltwater plume and density gradients. It is also slightly refined close to the ocean floor, see bottom panel of \fgref{fig:colorpics}. The finest mesh resolution close to the ice shelf is 3.5 \si{\meter}. This is comparable to the resolution in the equidistant, anisotropic grid used in the MITgcm simulations of \citep{Wiskandt} which consists of cells that are $3.33$ \si{\meter} in the vertical and $10$ \si{\meter} in the horizontal direction. 
\begin{figure}[H]
    \centering \includegraphics[width=1.\textwidth]{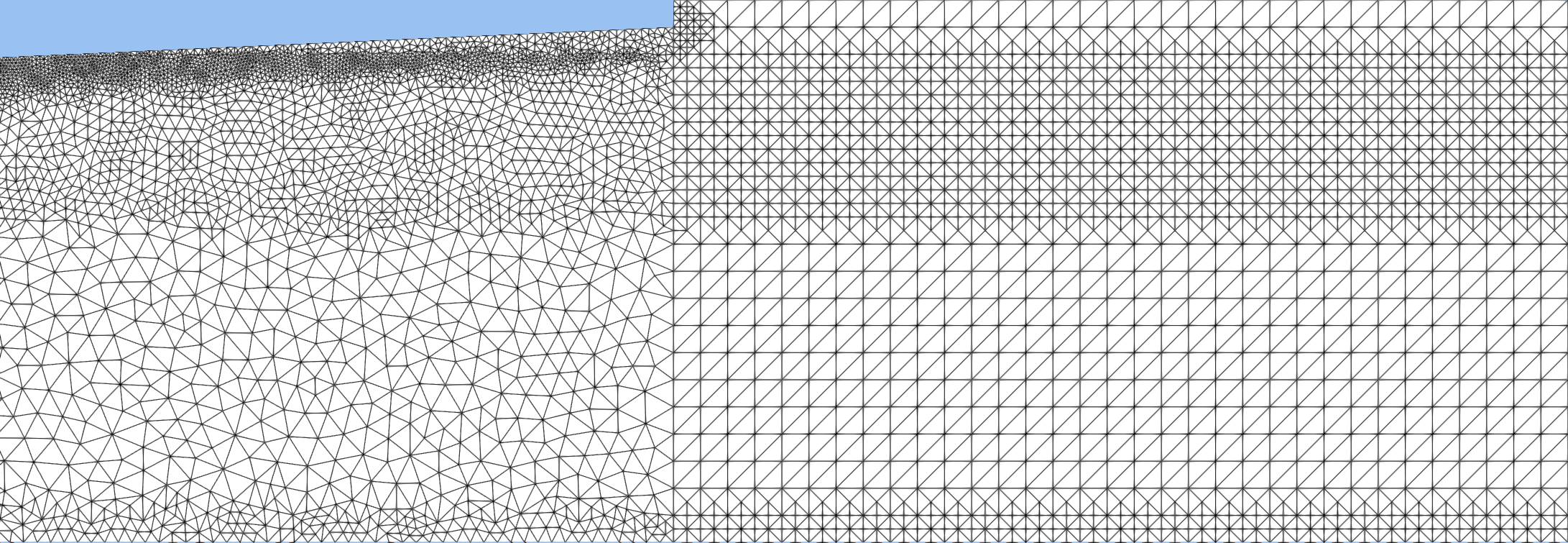}
  \caption{Unstructured mesh used for the Ryder simulation at $x \in [18.8\si{km},  21.7 \si{km}]$.}
      \label{fig:rydermesh}
\end{figure}

\subsubsection{Results and comparison to MITgcm}
In \fgref{fig:CirculationMelt} the time-average density distribution, streamlines and melt rates are shown. A buoyant layer of water forms close to the ice (typically referred to as a buoyant plume) and leaves the ice shelf when it reaches neutral buoyancy with respect to the surrounding water. The plume creates an estuarine circulation in the fjord, with an average overturning time $ \tau_0=11.9$ days, which is estimated as
\begin{equation}
    \tau_0(t) = 2 | \Omega | / \int_{y_{min}}^{y_{max}} | \bx \SCAL \bu(x=21 \si{km}, t) | \ud \bx.
\end{equation}

Due to circulation, warm Atlantic water (modeled by the restoring function at the open ocean boundary) flows along the ocean floor towards the glacier, causing it to melt further. The new open ocean boundary conditions, i.e. the modified directional do-nothing condition with a hydrostatic forcing \eqref{eq:ddn} let water flow in and out of the fjord. 

The melt rate as computed by the three-equation formulation is on average $+17.95$ \si{\meter \per yr}, meaning that the contribution to the mass balance of the Ryder glacier from submarine basal melting is negative. The melt rate increases with the $x$-coordinate, until the point where the plume detaches and the melt rate vanishes with the velocity (see \eqref{eq:3eqs} and \eqref{eq:gamma_tracer}
). The computed melt rate is similar to but slightly higher than in \citep{Wiskandt} (17.36 \si{\meter \per yr}), despite the no-slip conditions at the bottom and top of the domain (\citep{Wiskandt} use quadratic drag at the ice-sheet). The higher melt rate can probably be attributed to the residual viscosity sub-grid model, which adds less artificial viscosity than the more diffusive eddy viscosity model of MITgcm. This results in more fine structures and higher variation in the solution, see \fgref{fig:colorpics}. The temperature and salinity profile seen in  \fgref{fig:vertprofs} are thus less impacted by artificial smearing as compared to Fig 1d of \citep{Wiskandt}, so that the halocline and thermocline are closer to the initial state than that of \citep{Wiskandt}. The velocity profile is also less diffused with a higher peak velocity, see the top panel of \fgref{fig:vertprofs}, where the peak jet velocity is above $0.1$ \si{\meter \per yr}, while the velocity of \citep{Wiskandt}, as seen in their Fig. 3, is around $0.075$ \si{\meter \per yr}. We conclude that due to the new viscosity sub-grid model, the entrainment and mixing of ambient water into the buoyant plume is weaker than in the MITgcm simulation and that the plume water close to the ice in the FEM model, therefore, is more buoyant and hence faster, leading to higher melt rates. As excess diffusion and entrainment are important questions in modeling ice sheet-ocean interactions (see Discussion in \cite{Wiskandt}), future work is needed to quantify this in more detail and make recommendations for the climate modeling community.

\Fgref{fig:timeseries} shows the evolution of normalized kinetic energy $E_K(t)$, integrated temperature difference as compared to the initial values $\Delta T(t)$, and average melt rate $M_b(t)$, computed as

\begin{equation}
\begin{split}
    M_b = \int_{\Gamma_{ice}} m_b ds, \quad E_k(t) = \frac{ \| \bu(t) \|^2 - \hat{t}^{-1} \int_0^{\hat{t}} \| \bu(t) \|^2 dt }{ \text{std}( \| \bu(t) \|^2) },  \quad
    \Delta T(t) = \frac{\int_\Omega T(t) \ud \bx - \int_\Omega T(0) \ud \bx }{| \Omega | } ,
\end{split}
\end{equation}
where $\text{std}(f)$ denotes the standard deviation of $f$. These should be compared to Fig. B1 in \citep{Wiskandt}. The higher peak velocity of our simulations explains why we get a lower overturning timescale $\tau_0=11.9$ days compared to the $27$ days of the MITgcm model \citep{Wiskandt}, while the less diffused temperature explains the lower integrated temperature change $\Delta T$ of $-8\cdot 10^{-3}$ \si{\celsius} compared to the $-2.5 \cdot 10^{-2}$  \si{\celsius} of \citep{Wiskandt}.

\begin{figure}[H]
  \centering
    \begin{subfigure}{1.0 \textwidth}
    \centering
  \includegraphics[width=0.6\textwidth]{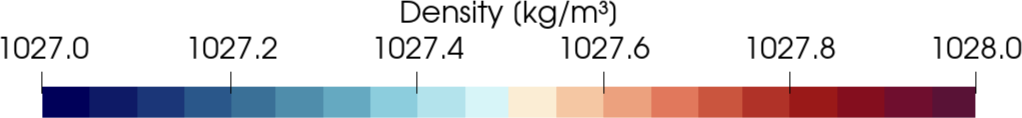}
  \includegraphics[width=1.\textwidth]{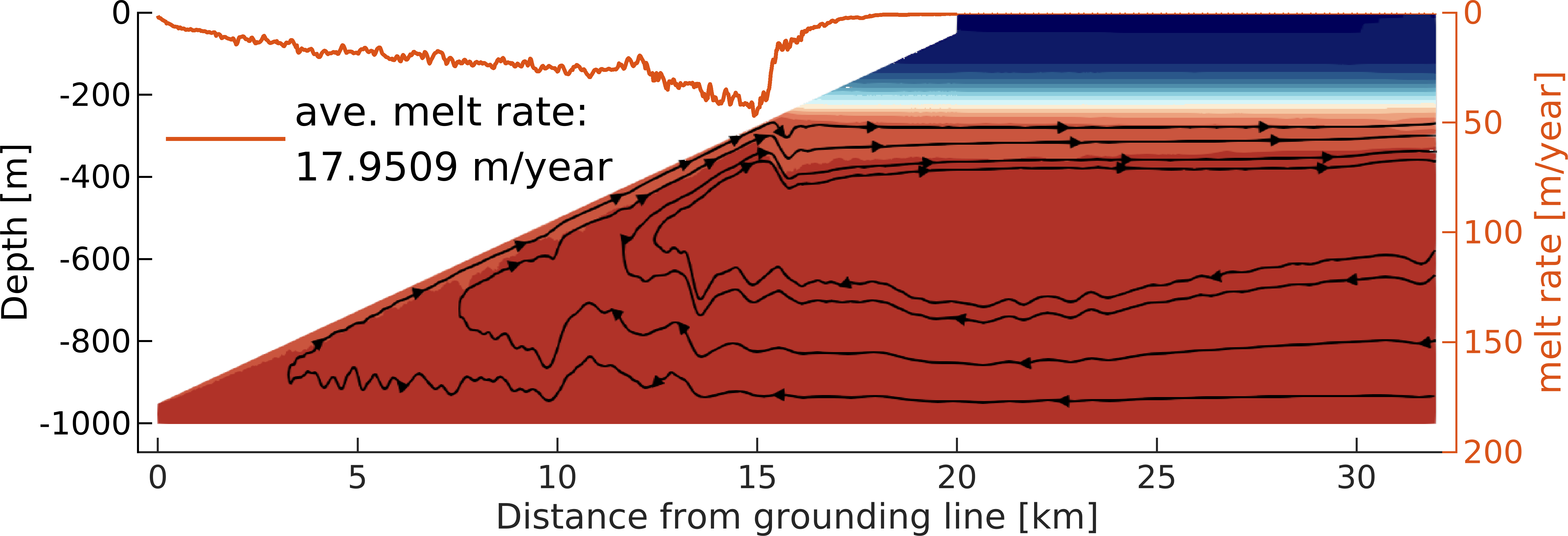}
  \end{subfigure}
  \caption{Time-averaged melt rate on the ice-shelf (red line) and time-averaged streamlines (black) superimposed on the density profile. The averages are taken from day 10 until the simulation is ended ($\approx 35$ days).}
      \label{fig:CirculationMelt}
\end{figure}%

\newcommand \leftTrim{5}
\newcommand \rightTrim{5}
\newcommand \sizeFig{1.0}

  \begin{figure}[H]
    \centering
  \begin{subfigure}{\sizeFig \textwidth}
    \centering
    \includegraphics[trim={\leftTrim cm 0 \rightTrim cm 18cm},clip,width=0.85\textwidth]{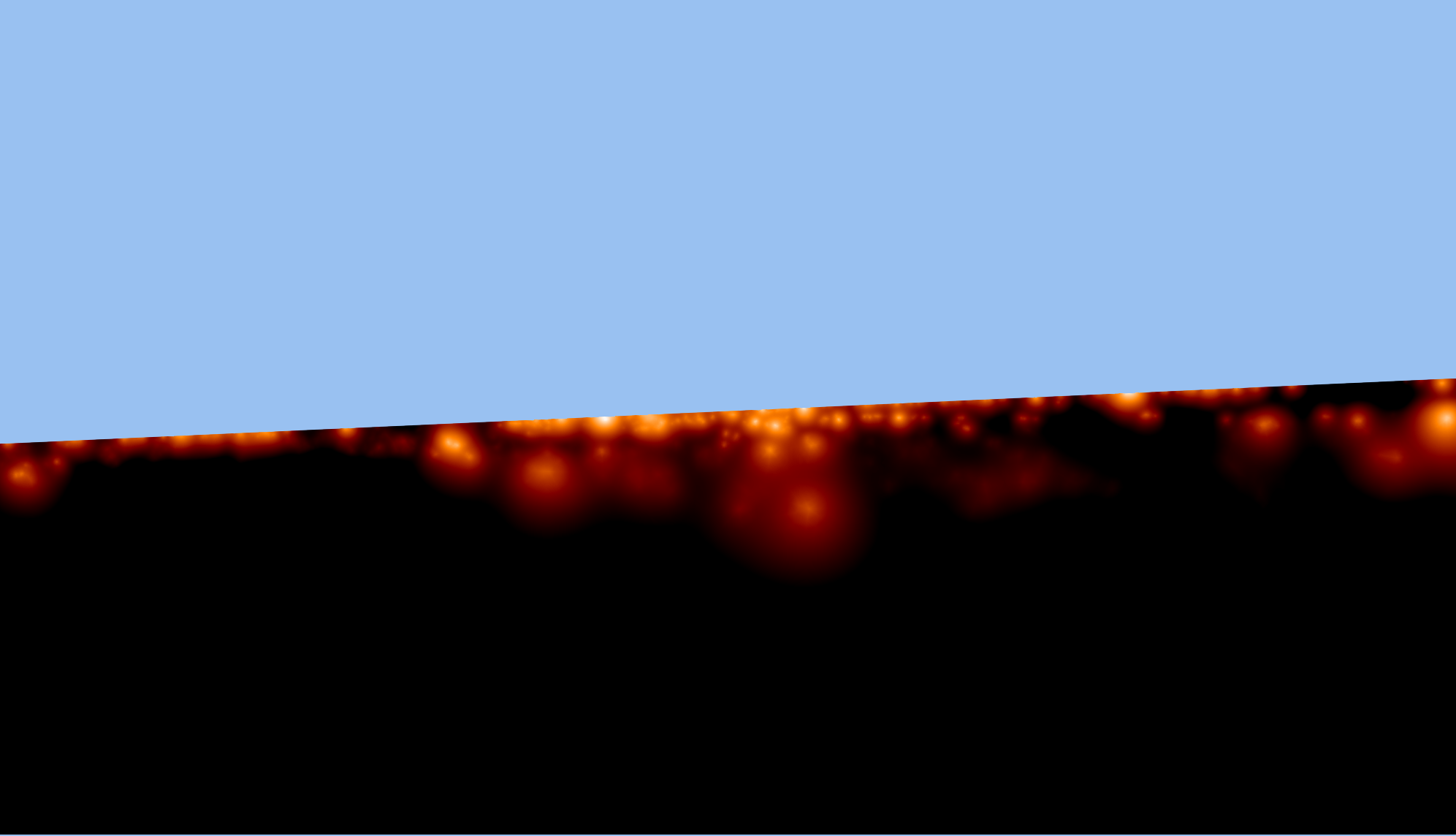}
    \includegraphics[width=0.105\textwidth]{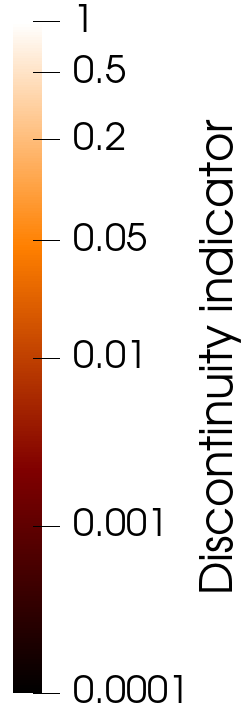}
  \end{subfigure} \\
  \centering
  \begin{subfigure}{\sizeFig \textwidth}
    \centering
    \includegraphics[trim={\leftTrim cm 0 \rightTrim cm 18cm},clip,width=0.85\textwidth]{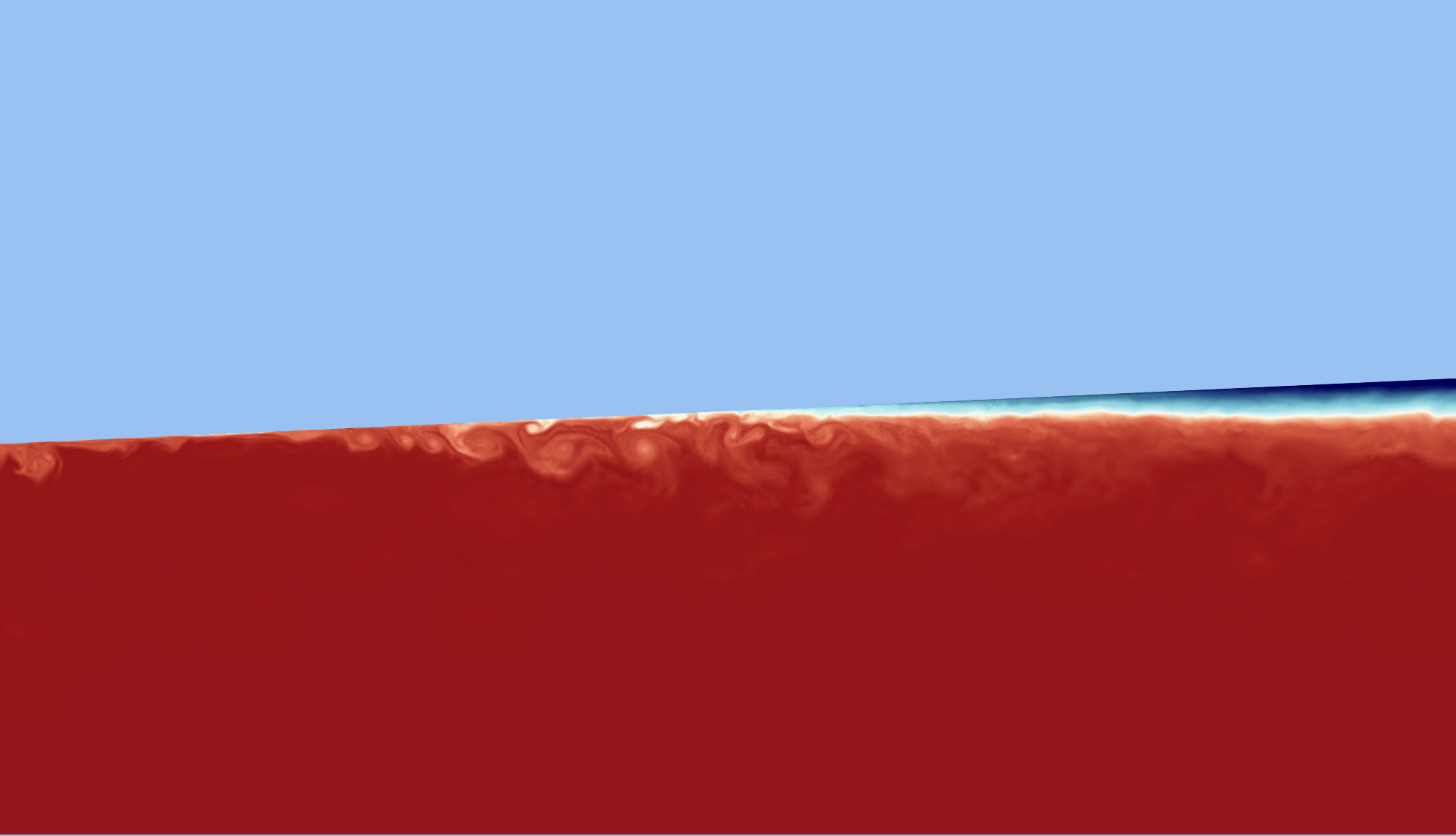}
    \centering
    \includegraphics[width=0.105\textwidth]{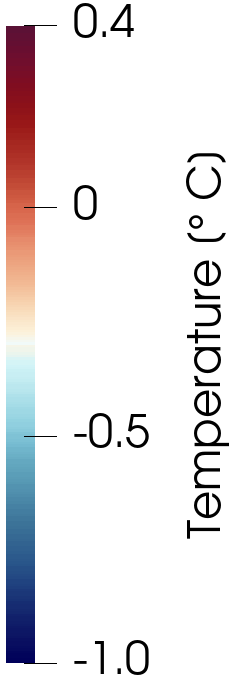}
  \end{subfigure} \\
  \centering
  \begin{subfigure}{\sizeFig \textwidth}
    \centering
    \includegraphics[trim={\leftTrim cm 0 \rightTrim cm 18cm},clip,width=0.85\textwidth]{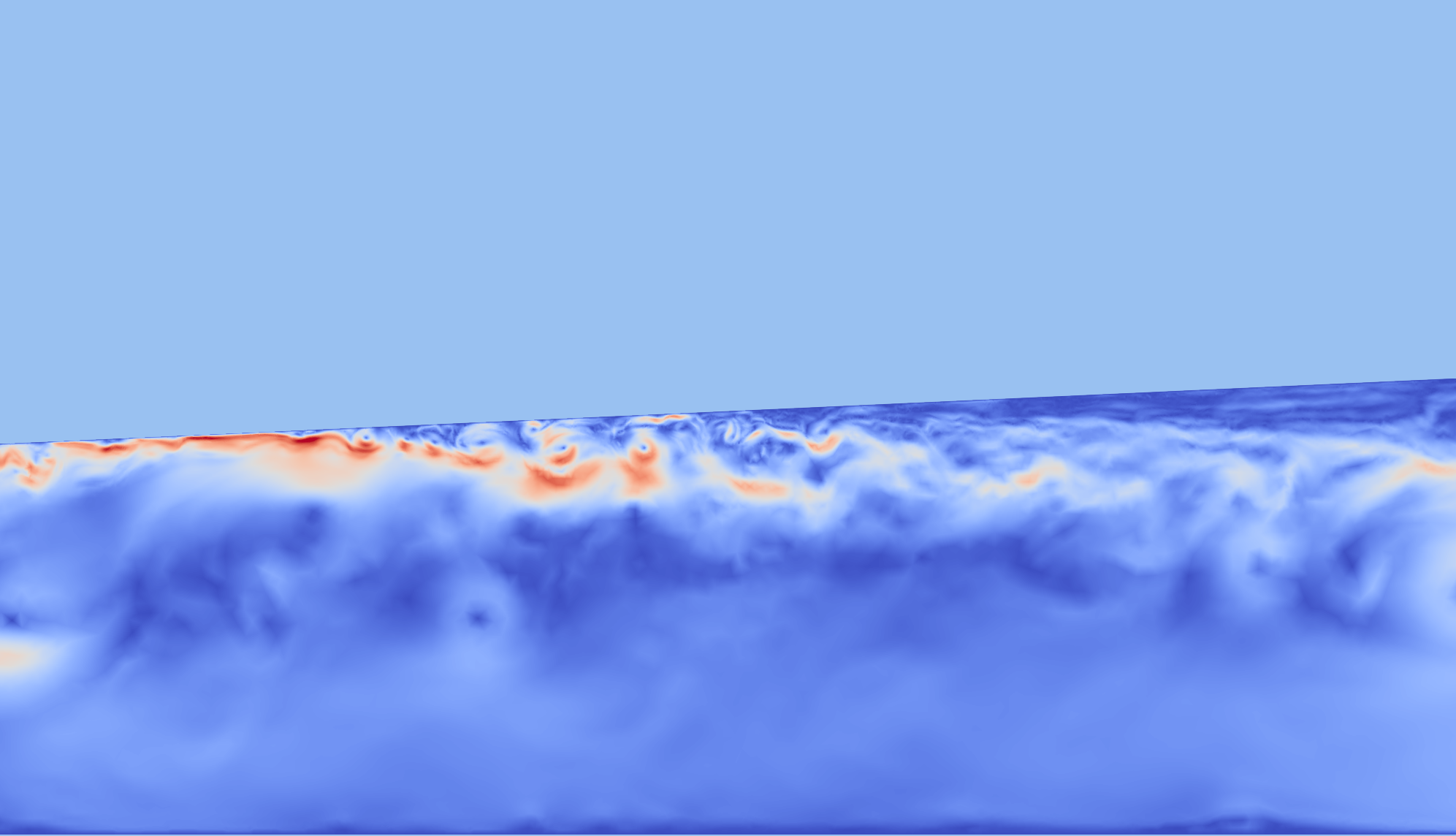}
    \centering
    \includegraphics[width=0.105\textwidth]{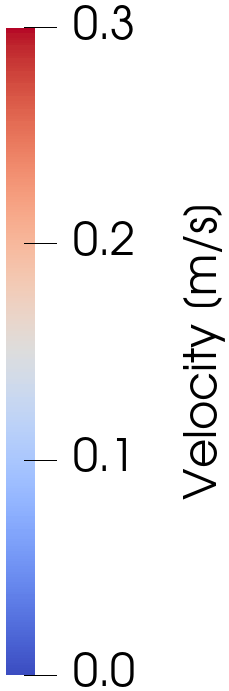}
  \end{subfigure} \\
  \begin{subfigure}{\sizeFig \textwidth}
    \centering
    \includegraphics[trim={\leftTrim cm 0 \rightTrim cm 18cm},clip,width=0.85\textwidth]{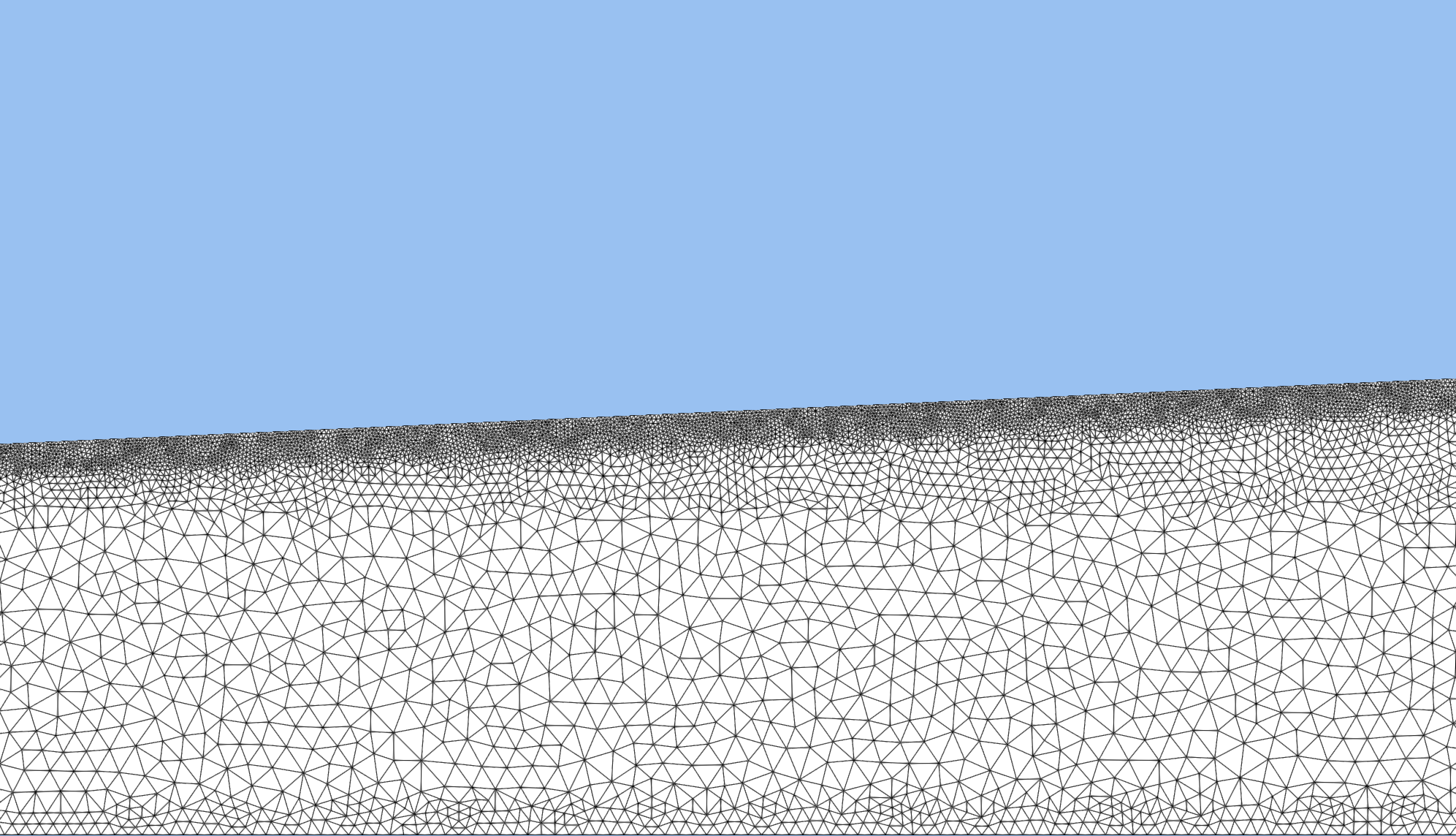}
    \includegraphics[width=0.105\textwidth]{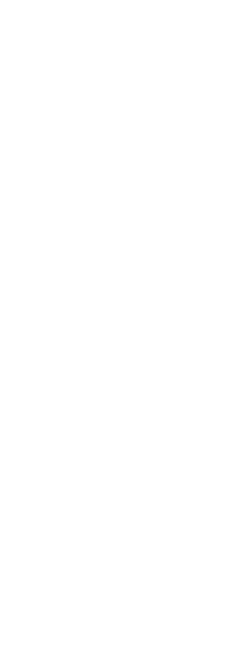}
  \end{subfigure}
  \caption{From top to bottom, snapshots of the discontinuity indicator for temperature in logarithmic scale (first), temperature (second), velocity magnitude (third) at $t =$ 10 days at the place where most melting occurs ($x \approx 15$ km). The unstructured mesh (bottom) is locally refined close to the ice shelf with $h \approx 3.5 \si{m}$. An animation is available in the supplementary material online.} 
   
    \label{fig:colorpics}
\end{figure}
\begin{figure}[H]
  \begin{subfigure}{1.0 \textwidth}
  \centering
  \includegraphics[width=0.5\textwidth]{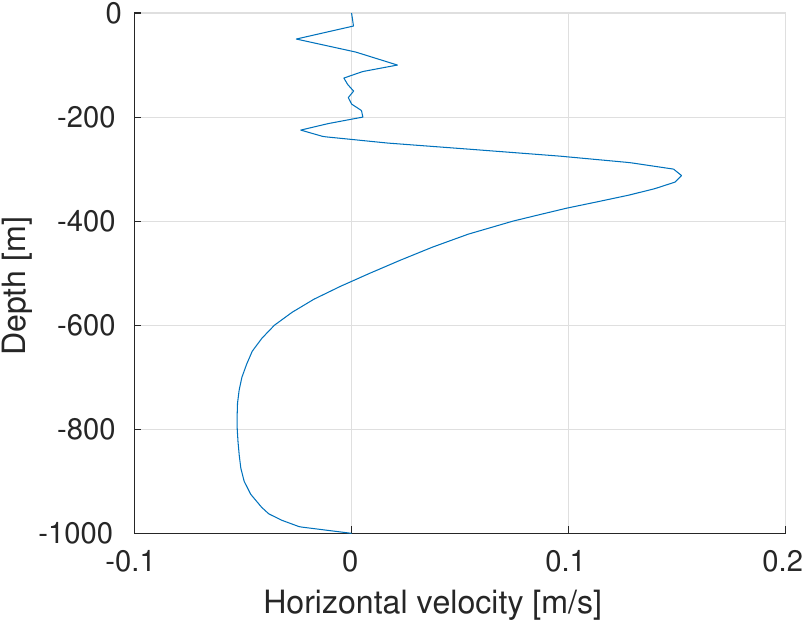}
  \caption{}
 \end{subfigure}
  \begin{subfigure}{0.5 \textwidth}
  \centering
  \includegraphics[width=1.0\textwidth]{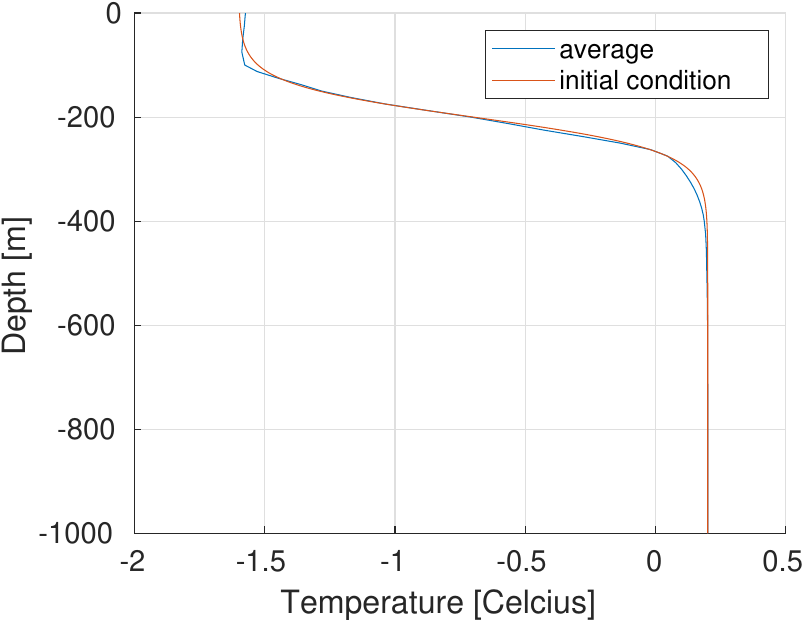}
  \caption{}
  \end{subfigure}
  \begin{subfigure}{0.5 \textwidth}
  \centering
  \includegraphics[width=1.0\textwidth]{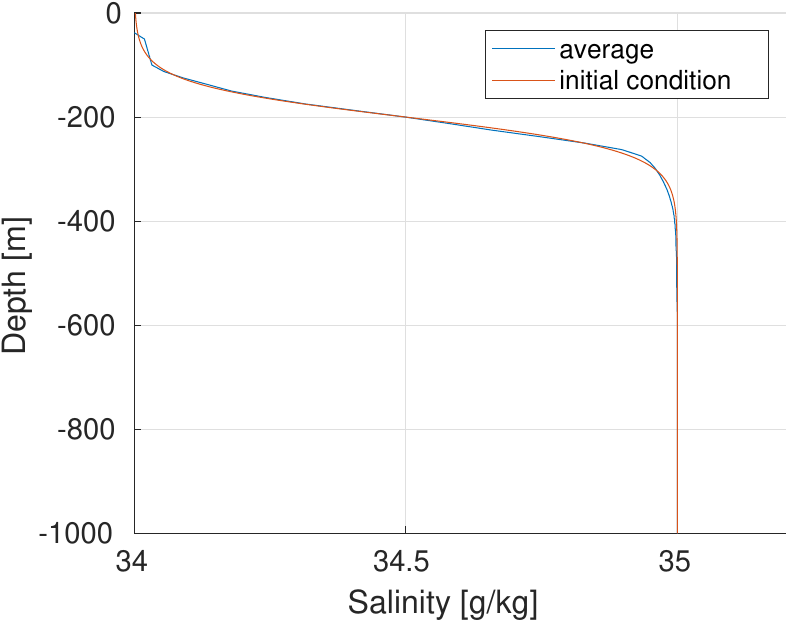}
  \caption{}
  \end{subfigure}  
  \caption{Time-averaged horizontal velocity (a), temperature (b) and salinity (c) at $x =$ 21 km as a function of depth. The average is taken from day 10 until the simulation is ended ($\approx 35$ days). Initial and open-ocean boundary condition profiles are also displayed in (b) and (c).}
    \label{fig:vertprofs}
\end{figure} %

\begin{figure}[H]
    \begin{subfigure}{1.0 \textwidth}
    \centering
  \includegraphics[width=1.0\textwidth,height=0.2\textheight]{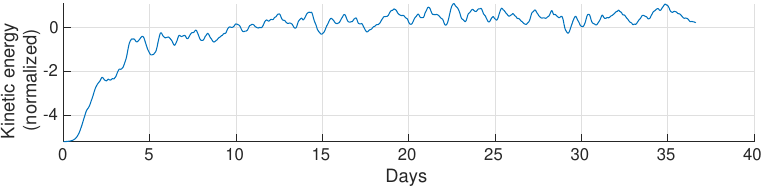}
    \end{subfigure}
   \begin{subfigure}{1.0 \textwidth}
    \centering
      \includegraphics[width=1.0\textwidth,height=0.2\textheight]{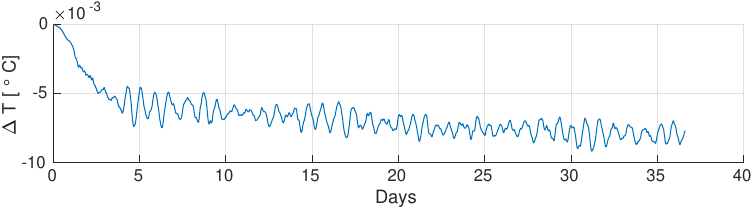}
        \end{subfigure}
 
  \begin{subfigure}{1.0 \textwidth}
    \centering
      \includegraphics[width=1.0\textwidth,height=0.2\textheight]{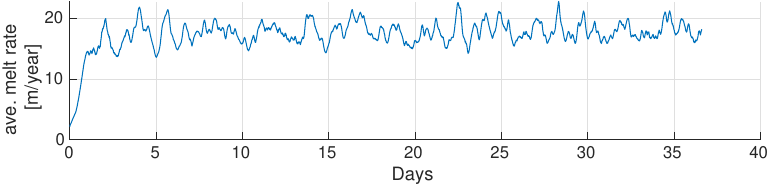}
        \end{subfigure}
        \caption{From top to bottom, normalized kinetic energy, integrated temperature change (compared to the initial state) and melt rate as functions of model days. For visualization purposes, a moving average of 0.1 days as filter width was applied.}
          \label{fig:timeseries}
\end{figure}



\section{Conclusion and further work}
This paper presents a continuous Galerkin finite element approximation of the Boussinesq equations suitable for variable density flow in general and ocean modeling and ice-ocean interaction in particular. The new methods are:
\begin{itemize}
\item A new formulation that allows for total energy conservation. The new consistent formulation of the Boussinesq system in combination with the special formulation of the tracer \eqref{eq:tracer_conservative} is shift-invariant and tracer mass, kinetic energy, potential energy, squared tracer density, momentum and angular momentum conserving (SI-MEEDMAC)
\begin{equation}
  \p_t \bu + \bu \SCAL \GRAD \bu + (\GRAD \bu) \bu + (\DIV \bu) \bu   = - \GRAD P + \DIV \l( \nu  \l( \GRAD \bu + (\GRAD \bu)^\top \r) \r) - \frac{\delta \rho g}{\rho_0} \Hat{\by} + \frac{1}{2} \GRAD \l( \frac{\delta \rho}{\rho_0} g y \r) ,
\end{equation}
with modified pressure $P = \frac{1}{\rho_0} p + g y - \frac{1}{2} \bu \SCAL \bu + \frac{1}{2} \frac{\delta \rho}{\rho_0} g y $. The formulation combines the EMAC formulation with a new force term via $\frac{1}{2} \GRAD \l( \frac{\delta \rho}{\rho_0} g y \r)$ that enables the FEM to conserve potential energy. This marks an improvement over the EMAC formulation for constant density flow \citep{Charnyi2017} and the SI-MEDMAC formulation for variable density flow \citep{lundgren2024fully} since now potential energy is also conserved.
\item A new sub-grid scale model that introduces less artificial diffusion compared to the eddy-viscosity models typically used in ocean models. The sub-grid scale model is based on a new symmetric viscosity tensor
\begin{equation}
    \sqrt{\bbnu_h}    \l( \GRAD \bu_h + (\GRAD \bu_h)^\top \r)  \sqrt{\bbnu_h}  , 
\end{equation}
which we show conserves angular momentum and dissipates kinetic energy. The method uses dynamic tensor-based viscosity as a sub-grid scale model with coefficients constructed based on the PDE residual. 
\item An open ocean boundary condition \eqref{eq:ddn_modified} that allows for inflow and outflow while conserving energy.
\end{itemize}
Together, the methods increase accuracy compared to standard methods and allow for the usage of unstructured meshes in glacier-fjord modeling. It compares well to the results of high-resolution MITgcm simulations \citep{Wiskandt} of the Sherard Osborn fjord, northern Greenland while allowing for coarser meshes in areas away from the ice and plume and avoiding overly diffusive velocity, temperature and salinity solutions. Future work involves the inclusion of the Coriolis force, investigating more realistic melt models and considering more realistic ice-sheet geometries. The last point involves developing stabilization techniques especially suitable for meshes with bad aspect ratios.


\section{Acknowledgments}
Lukas Lundgren and Christian Helanow were funded by the Swedish e-Science Research Centre (SeRC). 

Jonathan Wiskandt has been supported by the Faculty of Science, Stockholm University (grant no. SUFV-1.2.1-0124-17). 
The computations were enabled by resources in projects NAISS 2024/5-210 and NAISS 2023/5-185 provided by the National Academic Infrastructure for Supercomputing in Sweden (NAISS), partially funded by the Swedish Research Council through grant agreement no. 2022-06725.

\appendix

\section{Symmetric stabilization}
\label{Sec:symmetric_stabilization}
The purpose of this section is to show that \eqref{eq:stabilized fem} is equivalent to \eqref{eq:stabilized fem symmetric}. We start with the tracer update \eqref{eq:tracer_stabilized}.

Setting $\bw = \proj \GRAD w_h$ inside \eqref{eq:projection_LPS}, we obtain
\begin{equation}
    ( \bbkappa_{h,vms} \proj \grad \phi_h, \proj \GRAD w) = (\bbkappa_{h,vms} \grad \phi_h, \proj \GRAD w)
\end{equation}
which is equivalent to
\begin{equation} \label{eq:step_symmetric}
  (\bbkappa_{h,vms} ( \GRAD \phi_h - \proj \grad \phi_h)  , - \proj \GRAD w) = 0.
\end{equation}%
Adding \eqref{eq:step_symmetric} to \eqref{eq:tracer_stabilized} gives \eqref{eq:tracer_stabilized_symmetric}. %

Next, we consider the momentum equations. Using that the contraction between a symmetric and anti-symmetric matrix with zero diagonal is zero \citep[Ch 11.2.1]{Larson_2013}, \ie $(A + A^\top ): (B - B^\top) = 0$ where $A$ and $B$ are square matrices, one can show that
\begin{equation} \label{eq:sym}
\begin{split}
\l(  \sqrt{ \bbnu_h }  \l( \GRAD \bu_h + \l( \GRAD \bu_h \r)^\top \r)  \sqrt{\bbnu_h} , \GRAD \bv \r) \\ 
= \frac{1}{2}  \l(  \sqrt{ \bbnu_h }  \l( \GRAD \bu_h + \l( \GRAD \bu_h \r)^\top \r)  \sqrt{\bbnu_h}, \GRAD \bv + (\GRAD \bv)^\top \r) \\ +  \frac{1}{2} \l(  \sqrt{ \bbnu_h }  \l( \GRAD \bu_h + \l( \GRAD \bu_h \r)^\top \r)  \sqrt{\bbnu_h}, \GRAD \bv - (\GRAD \bv)^\top \r) \\ = \frac{1}{2}  \l(  \sqrt{ \bbnu_h }  \l( \GRAD \bu_h + \l( \GRAD \bu_h \r)^\top \r)  \sqrt{\bbnu_h} , \GRAD \bv + (\GRAD \bv)^\top \r).
\end{split}
\end{equation}
Applying this to the momentum update yields
\begin{equation} \label{eq:LPS_proof_step}
  \begin{split}
  ( \p_t \bu_h + \bu_h \SCAL \GRAD \bu_h + (\GRAD \bu_h) \bu_h + (\DIV \bu_h) \bu_h , \bv ) -( P_h, \DIV \bv) + (\gamma_h \DIV \bu_h, \DIV \bv) \\ + \frac{1}{2} \l(   \nu    \l( \GRAD \bu_h + (\GRAD \bu_h)^\top \r) +   \sqrt{\bbnu_h} \l( \GRAD \bu_h + (\GRAD \bu_h)^\top \r)  \sqrt{\bbnu_h}, \GRAD \bv + \GRAD \bv^\top \r)   \\   + \frac{1}{2} \l( \sqrt{ \bbnu_{h,vms} }  \l( \GRAD \bu_h + (\GRAD \bu_h)^\top - \proj \GRAD \bu_h - (\proj \GRAD \bu_h)^\top  \r) \sqrt{ \bbnu_{h,vms} } , \GRAD \bv + \GRAD \bv^\top \r) & =  \\  -  \l(   \frac{g \delta \rho_h}{\rho_0} \Hat{\by} , \bv \r)  -  \frac{1}{2} \l(   \frac{\delta \rho_h  }{\rho_0} g y  , \DIV \bv \r) + \frac{1}{2} \l(  ( \bu_h \SCAL \bn )_- \bu_h , \bv \r)_{\Gamma_{oc}} + \frac{1}{2} \l(  ( \bu_h \SCAL \bn ) \bu_h , \bv \r)_{\Gamma_{oc}} \\   - (P_{hyd} , \bv \SCAL \bn )_{\Gamma_{oc}}     \quad  \forall \bv \in \bcalV_h, 
\end{split}
\end{equation}
Next, setting $ \polV = \proj \GRAD \bv$ inside \eqref{eq:projection_LPS_tensor} one can show that
\begin{equation} \label{eq:tensor_step1}
  ( \sqrt{\bbnu_{h,vms}} (\GRAD \bu_h - \proj \grad \bu_h)\sqrt{\bbnu_{h,vms}}  , - \proj \GRAD \bv) = 0,
\end{equation}
which, due to \eqref{eq:trace_transpose}, is equivalent to
\begin{equation} \label{eq:tensor_step2}
  \l(\sqrt{\bbnu_{h,vms}} ( \GRAD \bu_h^\top - (\proj \grad \bu_h)^\top) \sqrt{\bbnu_{h,vms}}  , - (\proj \GRAD \bv)^\top \r) = 0.
\end{equation}
Adding \eqref{eq:tensor_step1} and \eqref{eq:tensor_step2} to \eqref{eq:LPS_proof_step} gives \eqref{eq:momentum_stabilized_symmetric}.

\section{Algorithm to compute the discontinuity indicator} \label{appendix:discontinuity_indicator}

This algorithm is taken directly from \cite{lundgren_2024RV} where all of the steps are explained in more detail and are motivated more. 


\begin{enumerate}
  \item On each node $i$, compute the residuals 
  \begin{equation} 
    \begin{split}
    R_{\phi_h ,i}^n &:= \l| d_t(\phi_h^n) + \bu_h^n \SCAL\GRAD \phi_h^n - \DIV ( \kappa_\phi \GRAD \phi_h^n ) - \zeta(x) ( \phi_{res}(y) - \phi_h^n ) \r|_{i} ,   \\
    R_{\bu_h,i}^n &:=    \l \|   d_t (\bu_h^n)    + \bu_h^n \SCAL\GRAD  \bu_h^n + (\GRAD \bu_h^n) \bu_h^n   + \GRAD P_h^n - \nu \DIV \l(  \GRAD \bu_h^n + (\GRAD \bu_h^n)^\top \r) + \frac{g \delta \rho_h}{\rho_0} \Hat{\by}^n - \frac{1}{2} \GRAD \l( \frac{\delta \rho_h}{\rho_0} g y \r) \r \|_{\ell_2,i}  ,
    \end{split}
    \end{equation} and normalization functions 
    \begin{equation} 
      \begin{split}
      n_{\phi_h,i}^n &:= 
     \begin{dcases} n_{\text{loc}, \phi_h, i}^n , \quad &\text{if } \quad h_i \| \GRAD \phi_h^n \|_{{\ell_2},i}  > C_\text{flat},  \\
       \max \l(   h_i^{-1} \ n_{\text{glob}} ( \phi_h^n  \| \bu_h^n \|_{\ell_2} ) , \ n_{\text{loc}, \phi_h,i}^n \r)   , \quad &\text{otherwise,} 
     \end{dcases} \\
     n_{\bu_h,i}^n &:=
     \begin{dcases} 
     n_{\text{loc}, \bu_h,i}^n, \quad &\text{if } \quad  h_i \|  \GRAD \bu_h \|_{{\ell_2},i}  > C_\text{flat}, \\
     \max \l(  h_i^{-1} \ n_{\rm{glob}} (\|\bu_h^n\|_{\ell_2}^2), \ n_{\text{loc}, \bu_h, i}^n \r)   , \quad &\text{otherwise,} 
     \end{dcases} 
     \end{split}
     \end{equation}
     where 
     \begin{equation} 
      \begin{split}
      n_{\text{loc}, \phi_h,i}^n :=&   | d_t(\phi_h^n) |_{i} +  \| \bu_h^n \|_{\ell_2,i} \| \GRAD \phi_h^n \|_{\ell_2,i} + |\DIV ( \kappa_\phi \GRAD \phi_h^n )| + |\zeta(x) ( \phi_{res}(y) - \phi_h^n )|,  \\
      n_{\text{loc}, \bu_h,i}^n :=&     \|  d_t(\bu_h^n)   \|_{\ell_2,i} + 2 \|  \bu_h^n \|_{\ell_2,i} \|  \GRAD   \bu_h^n  \|_{\ell_2,i}  \\  &  + \| \GRAD P_h^n \|_{\ell_2,i}  + \l \|  \nu \DIV \l(  \GRAD \bu_h^n+ (\GRAD \bu_h^n)^\top \r) \r \|_{\ell_2,i}  + \l \| - \frac{g \delta \rho_h}{\rho_0} \Hat{\by}^n  \r  \|_{\ell_2,i} + \frac{1}{2}  \l \| \GRAD \l( \frac{\delta \rho_h}{\rho_0} g y \r) \r \|_{\ell_2,i}  , 
      \end{split}
      \end{equation}
where
\begin{equation}
        n_{\text{glob}}(w) :=  \frac{ \l| \max_{\Omega }  w - \min_{\Omega }  w \r|^2}{  |\max_{\Omega }  w - \min_{\Omega }  w|  + \varepsilon \| w \|_{L^\infty(\Omega \times (0,t_n))} }.   
        \end{equation}
      

  \item Solve the projection problem with smoothing 
  \begin{equation} 
    \l( \tilde{\ind}_{h,\phi,\bu}^n, w \r) + C_{\Delta} \l( h^2 \GRAD \tilde{\ind}_{h,\phi,\bu}^n, \GRAD w \r)  = \l( f_\text{RV} \l(  \frac{R_{\phi,\bu}^n}{n_{\phi,\bu}^n} \r) , w \r), \quad \forall w \in \calM_h, 
    \end{equation}
    where $f_{RV}$ is a user-defined activation function.
  \item Compute: \begin{equation}  
    \ind_{h,\phi,\bu}^n := \min(1,|\tilde{\ind}_{h,\phi,\bu}^n|).
    \end{equation}
  \end{enumerate}

\bibliographystyle{abbrvnat}
\bibliography{references}

\end{document}